\documentclass[11pt,english]{article}

\usepackage[margin = 2.5cm]{geometry}

\usepackage{amsthm}
\usepackage{amsmath}
\usepackage{amssymb}
\usepackage{amsfonts}
\usepackage{setspace}
\usepackage[hidelinks]{hyperref}
\usepackage{cleveref}

\usepackage{enumitem}
\setlist[itemize]{leftmargin=*}
\setlist[enumerate]{itemsep=3pt,topsep=3pt}
\setlist[enumerate,1]{label=\rm{(\alph*)}}

\usepackage{subcaption}
\usepackage{tikz}

\usepackage[square,sort,comma,numbers]{natbib} 
\theoremstyle{plain}

\newtheorem*{thm*}{Theorem}
\newtheorem{thm}{Theorem}
\Crefname{thm}{Theorem}{Theorems}

\newtheorem*{lem*}{Lemma}
\newtheorem{lem}[thm]{Lemma}
\Crefname{lem}{Lemma}{Lemmas}

\newtheorem*{claim*}{Claim}
\newtheorem{claim}[thm]{Claim}
\crefname{claim}{Claim}{Claims}
\Crefname{claim}{Claim}{Claims}

\newtheorem{prop}[thm]{Proposition}
\Crefname{prop}{Proposition}{Propositions}

\newtheorem{cor}[thm]{Corollary}
\Crefname{cor}{Corollary}{Corollaries}
\crefname{cor}{Corollary}{Corollaries}

\Crefname{conj}{Conjecture}{Conjectures}

\Crefname{qn}{Question}{Questions}

\newtheorem{obs}[thm]{Observation}
\Crefname{obs}{Observation}{Observations}
\crefname{obs}{Observation}{Observations}

\theoremstyle{definition}

\Crefname{prob}{Problem}{Problems}

\Crefname{defn}{Definition}{Definitions}

\theoremstyle{remark}

\Crefname{rem}{Remark}{Remarks}

\renewenvironment{proof}[1][]{\begin{trivlist}
\item[\hspace{\labelsep}{\bf\noindent Proof#1.\/}] }{\qed\end{trivlist}}

\newcommand{\eps}{\varepsilon}

\renewcommand{\P}{\mathcal{P}}

\newcommand{\N}{\mathbb{N}}
\newcommand{\Pb}[1]{\mathbb{P}\!\left[#1\right]}
\newcommand{\R}{\mathbb{R}}

\DeclareMathOperator{\RN}{RN}

\newcommand{\remove}[1]{}

\newcommand{\ceil}[1]{
    \lceil #1 \rceil
}
\newcommand{\floor}[1]{
    \lfloor #1 \rfloor
}

\newcommand{\bipham}{weakly-Hamiltonian}
\newcommand{\ham}{Hamiltonian}

\newcommand{\es}{\emptyset}
\newcommand{\sm}{\setminus}

\newcommand{\crossedges}[2]{${#1}\text{\,--\,}{#2}$}
\newcommand{\nearbroot}{almost-bipartite}
\newcommand{\farbroot}{far-from-bipartite}
\newcommand{\nearb}[1]{${#1}$-\nearbroot}
\newcommand{\farb}[1]{${#1}$-\farbroot}
\newcommand{\almostbal}[1]{${#1}$-balancing}

\newcommand{\Evrich}{\mathcal{E}_2}
\newcommand{\Evpoor}{\mathcal{E}_1}
\newcommand{\Erich}{E_{\mathrm{rich}}}
\newcommand{\Epoor}{E_{\mathrm{poor}}}
\newcommand{\cmin}{c_{\mathrm{min}}}
\newcommand{\rmax}{r_{\mathrm{max}}}

\newcommand{\liftG}{\bar{G}}
\newcommand{\liftA}{\bar{A}}
\newcommand{\partb}[1]{#1^{(1)}}
\newcommand{\partt}[1]{#1^{(2)}}
\newcommand{\Gsigma}{G_\sigma}
\newcommand{\Gsigmadi}{\vec{G}_\sigma}
\DeclareMathOperator{\disb}{imb}
\DeclareMathOperator{\capacity}{capacity}
\DeclareMathOperator{\val}{value}
\newcommand{\step}[2]{\vspace{6pt}\par\noindent{\bf Step #1:} {\it #2.}\par\vspace{6pt}}

\begin{document}

\title{Cycle partitions of regular graphs}
\author{
    Vytautas Gruslys\thanks{
		Email: 
		\texttt{vytautas.gruslys@gmail.com}.
    }
    \and
		Shoham Letzter\thanks{
		Department of Mathematics, 
		University College London, 
		Gower Street, London WC1E~6BT, UK. 
		Email: \texttt{s.letzter}@\texttt{ucl.ac.uk}. 
		Research was supported by Dr.\ Max R\"ossler, by the Walter Haefner Foundation and by the ETH Zurich Foundation.
	}
}
\date{}

\maketitle

\begin{abstract}

    \setlength{\parskip}{\medskipamount}
    \setlength{\parindent}{0pt}
    \noindent

	Magnant and Martin conjectured that the vertex set of any $d$-regular
	graph $G$ on $n$ vertices can be partitioned into $n / (d+1)$ paths
	(there exists a simple construction showing that this bound would be
	best possible). We prove this conjecture when $d = \Omega(n)$, improving
	a result of Han, who showed that in this range almost all vertices of
	$G$ can be covered by $n / (d+1) + 1$ vertex-disjoint paths. In fact,
	our proof gives a partition of $V(G)$ into cycles. We also show that, if
	$d = \Omega(n)$ and $G$ is bipartite, then $V(G)$ can be partitioned
	into $n / (2d)$ paths (this bound is tight for bipartite graphs).

\end{abstract}

\section{Introduction} \label{sec:intro} 
	
	Dirac's classical result states that every graph on $n \ge 3$ vertices
	with minimum degree at least $n/2$ contains a Hamilton cycle. This
	minimum degree condition is best possible, as there is no Hamilton cycle
	in the almost balanced complete bipartite graph $K_{\floor{(n-1)/2},
	\ceil{(n+1)/2}}$ nor in the graph obtained by overlapping two cliques,
	$K_{\floor{(n+1)/2}}$ and $K_{\ceil{(n+1)/2}}$, at a single vertex.
	While this means that Dirac's result cannot be extended to general
	graphs with minimum degree lower than $n/2$, such an extension may be
	possible if certain natural conditions are imposed on the graph. A very
	nice conjecture, posed independently by Bollob\'as \cite{bollobas} and
	H\"aggkvist (see~\cite{jackson}), stated that if $d \ge n/(t+1)$ then
	every $t$-connected $d$-regular graph on $n$ vertices is Hamiltonian.
	It is indeed natural to require the graph be regular so that imbalanced
	complete bipartite graphs are ruled out. Note that the case $t = 1$
	follows directly from Dirac's theorem. 
	
	The conjecture of Bollob\'as and H\"aggkvist has been resolved.  The
	case $t = 2$ was proved by Jackson~\cite{jackson}, following partial
	results of Nash-Williams~\cite{nash-williams}, Erd\H{o}s and
	Hobbs~\cite{erdos-hobbs}, and Bollob\'as and Hobbs
	\cite{bollobas-hobbs}. Jackson's result was strengthened slightly by
	Hilbig~\cite{hilbig}, who showed that there are only two extremal
	examples (that is, $2$-connected $d$-regular graphs with $d \ge n/3 - 1$
	and no Hamilton cycle), namely, the Petersen graph and the graph
	obtained by replacing one vertex of the Petersen graph by a triangle.
	Following a number of partial results by Fan \cite{fan}, Jung
	\cite{jung}, Li and Zhu \cite{li-zhu}, Broersma, van den Heuvel, Jackson
	and Veldman \cite{broersma-heuvel-jackson-veldman}, and Jackson, Li and
	Zhu \cite{jackson-li-zhu}, the case $t = 3$ was recently proved by
	K\"uhn, Lo, Osthus and Staden in two papers where they first obtained an
	asymptotic result \cite{kuhn-lo-osthus-staden-approx} and then the exact
	result (for large $n$) \cite{kuhn-lo-osthus-staden}.  This completed the
	picture regarding the Bollob\'as and H\"aggkvist conjecture, since the
	conjecture is false for $t \ge 4$, as was shown by Jung \cite{jung} and
	by Jackson, Li and Zhu \cite{jackson-li-zhu}. 

	A different direction was suggested by Enomoto, Kaneko and Tuza
	\cite{enomoto-kaneko-tuza}: rather than finding one Hamilton cycle, they
	were interested in finding a small collection of cycles that covers the
	vertex set. More precisely, they conjectured that the vertices of any
	$n$-vertex graph with minimum degree at least $d$ can be covered by at
	most $(n - 1) / d$ cycles, where edges are considered to be cycles on
	two vertices. Note that the case where $d = n/2$ is exactly Dirac's
	theorem. The bound $(n - 1) / d$ cannot be meaningfully lowered, since
	at least $\floor{(n - 1) / d}$ cycles are needed to cover the vertices
	of $K_{n-d, d}$ or of the graph obtained by taking one vertex of full
	degree and covering the other $n - 1$ vertices by $\floor{(n - 1) / d}$
	disjoint cliques, each of order at least $d$.  Following progress by
	Enomoto, Kaneko and Tuza \cite{enomoto-kaneko-tuza} and Kouider
	\cite{kouider}, this conjecture was proved by Kouider and Lonc
	\cite{kouider-lonc} and, much later, but independently, by Balogh,
	Mousset and Skokan \cite{balogh-mousset-skokan} for $d = \Omega(n)$.

	What if the cycles in the conjecture of Enomoto, Kaneko and Tuza are
	required to be vertex-disjoint? In this case imbalanced bipartite graphs
	are again problematic, and so it makes sense to consider regular graphs.
	Magnant and Martin \cite{magnant-martin} conjectured that the vertices
	of any $n$-vertex $d$-regular graph can be covered by at most $n/(d+1)$
	vertex-disjoint paths; this bound is tight as can be seen by taking a
	disjoint union of cliques of order $d+1$ (and, possibly, a larger
	$d$-regular graph on the remaining $d+1$ to $2d+1$ vertices).  They
	proved this conjecture for $d \le 5$ and Han \cite{han} proved that, if
	$d = \Omega(n)$, then all but $o(n)$ vertices can be covered by at most
	$n/(d+1)+1$ paths. It does not seem critical that Magnant and Martin
	stated their conjecture for paths and not for cycles, because (at least
	in dense graphs) typical methods that give path partitions tend to give
	cycle partitions just as well. In this paper we prove Magnant and
	Martin's conjecture when $d = \Omega(n)$ and, indeed, our proof gives a
	partition into cycles. 

	\begin{thm} \label{thm:main}
		For every $\cmin > 0$ there exists $n_0$ such that if $G$ is a
		$d$-regular graph on $n$ vertices, where $n \ge n_0$ and $d \ge \cmin
		n$, then $V(G)$ can be partitioned into at most $n/(d+1)$ cycles.
	\end{thm}

	We also obtain an analogous result for bipartite graphs, but this time
	we only establish the existence of a path partition. The reason why our
	proof does not work for cycles seems to be technical rather than
	essential: we do believe that the same approach can give a proof for
	cycles, provided that some of our lemmas, including the main lemma of
	\Cref{sec:balancing}, are expanded with further technical conditions.
	However, to maintain the readability of this paper, we do not pursue
	this marginally stronger result.
	
	\begin{thm} \label{thm:main-bip}
		For every $\cmin > 0$ there exists $n_0$ such that if $n \ge n_0$,
		$d \ge \cmin n$ and $G$ is a $d$-regular bipartite graph on $n$
		vertices, then $V(G)$ can be partitioned into at most $n / (2d)$
		paths.  
	\end{thm}

	\Cref{thm:main-bip} improves a result of Han~\cite{han}, who proved that
	all but $o(n)$ vertices can be covered by at most $n / (2d)$
	vertex-disjoint paths. The bound $n / (2d)$ can be seen to be tight by
	taking a disjoint union of $\floor{n / (2d)}$ $K_{d,d}$'s (possibly,
	replacing one of them by a slightly bigger $d$-regular bipartite graph,
	making sure that exactly $n$ vertices are used).

	In the following section we outline the proofs and the structure of the
	rest of the paper.  

\section{Overview} \label{sec:stuff}

	\subsection{Outline of the proof}

        Our plan for proving \Cref{thm:main} is as follows. (The proof of
        \Cref{thm:main-bip} is similar and, in fact, slightly simpler.) First,
        we partition the vertices into a small number of parts, which we call
        clusters, that are well-connected and such that there are few edges with
        ends in different clusters (this is made precise in
        \Cref{lem:good-partition}).  K\"uhn, Lo, Osthus and Staden
        \cite{kuhn-lo-osthus-staden,kuhn-lo-osthus-staden-approx} used a similar
        partition. Moreover, the clusters
        in our partition can be shown to be \emph{robust expanders}, a term that
        was introduced by K\"uhn, Osthus and Treglown
        \cite{kuhn-osthus-treglown} and has since proved to be very useful (see,
        for instance,
        \cite{kuhn-osthus,kuhn-mycroft-osthus,kuhn-mycroft-osthus-exact}). 

        We zoom in on each cluster: ideally, we would like each one of them to
        be Hamiltonian and remain Hamiltonian after the removal of any small set
        of vertices. We establish this fact about all clusters that cannot be
        made bipartite by removing a small number of edges. However, the
        statement may fail for other clusters; for example, an imbalanced
        bipartite graph may appear as a cluster, and it is certainly not
        Hamiltonian. For clusters that are almost bipartite we establish a more
        technical statement: they become Hamiltonian after the removal of any
        small set of vertices that balances its two sides. This is done in
		\Cref{lem:rob-ham}, whose proof follows relatively easily from a result in \cite{kuhn-osthus-treglown}. 

        Up to this point our argument mostly follows the strategy in
        \cite{kuhn-lo-osthus-staden}. Our main new ideas are in the proof of the
        next lemma, \Cref{lem:good-linear-forest}, in which we construct a small
        linear forest whose removal balances the clusters that are almost
        bipartite. A similar linear forest was constructed by K\"uhn, Lo, Osthus
        and Staden \cite{kuhn-lo-osthus-staden,kuhn-lo-osthus-staden-approx}.
        However, their approach was more ad hoc and relied on the number of
        clusters being small (namely, at most five), whereas here
        this number can be arbitrarily large.
        
        Upon the removal of the interior vertices of this linear forest, the
        clusters become Hamiltonian; in them we pick Hamilton paths that attach
        to the leaves of the linear forest. This ensures that the paths in the
        linear forest can be concatenated with the Hamilton paths in the
        clusters.  The result is a small family of vertex-disjoint paths --
        containing no more paths than there are clusters  --
        that covers the whole graph. By doing this step carefully, we ensure
        that each path in the family starts and ends at adjacent vertices,
        which means that this family is in fact a family of cycles.

	\subsection{Key lemmas}

        In this subsection we give some definitions and state
        \Cref{lem:good-partition,lem:rob-ham,lem:good-linear-forest}.

		Here and later, we freely use standard definitions in graph theory:
		$e(H)$ denotes the number of edges of a graph $H$ and, for disjoint
		sets $X, Y \subset V(H)$, we denote by $H[X, Y]$ the graph with vertex
		set $X \cup Y$ whose edges are the \crossedges{X}{Y} edges of $H$
		(that is, those edges of $H$ with one end in $X$ and one in $Y$).
		Let $G$ be a graph on $n$ vertices.
        A \emph{cut} of a set $A \subset V(G)$ is a partition $\{X, Y\}$ of $A$,
        where $X$ and $Y$ are both non-empty.  We say that a cut $\{X, Y\}$ is
        \emph{$\alpha$-sparse} if $e(G[X, Y]) \le \alpha|X||Y|$. 
        We say that a set $A \subset V(G)$ is \emph{\nearb{\alpha}} if there
        exists a partition $\{X, Y\}$ of $A$ such that $G[A]$ has at most
        $\alpha n^2$ edges that are not \crossedges{X}{Y} edges. Otherwise,
        we say that $A$ is \emph{\farb{\alpha}}.

		The following lemma, which is very similar to a result from \cite{kuhn-lo-osthus-staden-approx}, partitions the vertices of $G$ into a small number
        of well-behaved sets, which we call \emph{clusters}.
			
        \begin{lem} \label{lem:good-partition}
			Let $\cmin \in (0,1)$ and $n_0 \in \N$ be such that $1/n_0 \ll \cmin$.
			Let $G$ be a $d$-regular graph on $n$ vertices, where $n \ge n_0$ and
			$d \ge \cmin n$.  Then there exist parameters $r \le
			1/\cmin$ and $\eta, \beta, \gamma, \zeta, \delta$, where
			$1/n_0 \ll \eta \ll \beta \ll \gamma \ll \zeta \ll \delta \ll \cmin$, and a partition
			$\{A_1, \ldots, A_r\}$ of $V(G)$ into non-empty sets satisfying the
			following properties: 
            \begin{enumerate} 
                \item \label{itm:few-cross-edges}
                    $G$ has at most $\eta n^2$ edges with ends in different
                    $A_i$'s; 
                \item \label{itm:min-deg}
                    for each $i \in [r]$, the minimum degree of $G[A_i]$ is at
                    least $\delta n$;
                \item \label{itm:no-small-cuts}
                    for each $i \in [r]$, $A_i$ has no $\zeta$-sparse
                    cuts;
                \item \label{itm:close-far-bip}
                    for each $i \in [r]$, $A_i$ is either \nearb{\beta}
                    or \farb{\gamma}.
            \end{enumerate}
        \end{lem}

        The meaning of the symbol $\ll$ requires some clarification. Every
        expression of the form $a \ll b$ should be read as `$a$ is much less
        than $b$'. Formally, it means that $a < \Phi(b)$ where $\Phi : (0,1] \to
        (0,1]$ is a hidden increasing function associated to that particular
        expression.  The hidden functions depend only on the constant $\cmin$,
        and they can be worked out by carefully following the forthcoming
        arguments.  We shall not mention these function again; instead, we shall
        implicitly assume that, as the variable approaches $0$, they decrease
        sufficiently fast to make our calculations work.

		We remark that the statement of \Cref{lem:good-partition} is somewhat unusual in that, given $n_0$, $\cmin$ and $G$ as in the lemma, the conclusion holds for \emph{some} choice of parameters $\eta, \beta, \gamma, \zeta, \delta$, with $1/n_0 \ll \eta \ll \beta \ll \gamma \ll \zeta \ll \delta \ll \cmin$, but not for every choice of such parameters. In particular, the correct choice for parameters depends on the graph $G$.

        Given a graph $G$ on $n$ vertices, a set $A \subset V(G)$ is called \emph{$\xi$-\ham} if, for any subset
        $W$ of size at most $\xi n$ and any pair of distinct vertices $x, y
        \in A \setminus W$, there is a Hamilton path in $G[A \setminus W]$ with
        ends $x, y$.  Given a partition $\{X, Y\}$ of $A$, we say that $A$ is
        \emph{$\xi$-\bipham{} with respect to $\{X, Y\}$} if, for any subset
        $W$ of size at most $\xi n$ that satisfies $|X \setminus W| = |Y
        \setminus W|$ and any vertices $x \in X \setminus W,\; y \in Y \setminus
        W$, there is a Hamilton path in $G[A \sm W]$ with ends $x, y$.

        The following lemma shows that clusters are \ham{} if they are far from
        being bipartite and \bipham{} if they are almost bipartite.

        \begin{lem} \label{lem:rob-ham}
			Let $\cmin \in (0,1)$ and $n \in \N$, and let $\eta, \beta, \xi,
			\gamma, \zeta, \delta$ be real numbers satisfying $1/n \ll \eta \ll
			\beta \ll \xi \ll \gamma \ll \zeta \ll \delta \ll \cmin$.  Let $G$
			be a $d$-regular graph on $n$ vertices, where $d \ge \cmin n$, and
			suppose that $A \subset V(G)$ satisfies the following properties.
            \begin{enumerate}
				\item \label{itm:few-cross-edges-A}
                    there are at most $\eta n^2$ edges in $G$ with exactly one
                    end in $A$;
				\item \label{itm:min-deg-A}
                    $G[A]$ has minimum degree at least $\delta n$;
				\item \label{itm:no-small-cuts-A}
                    $A$ has no $\zeta$-sparse cuts;
				\item \label{itm:close-far-bip-A}
                    $A$ is either \nearb{\beta} or \farb{\gamma}.
            \end{enumerate}
            If $A$ is \farb{\gamma}, then $A$ is $\xi$-\ham; if $A$ is
            \nearb{\beta}, then it is $\xi$-\bipham{} with respect to any
            partition $\{X, Y\}$ of $A$ that maximises the number of
            \crossedges{X}{Y} edges.
        \end{lem}

		When presented with a partition into well-behaved clusters, the next
		lemma produces a collection of vertex-disjoint paths that balances
		the clusters.

		\begin{lem} \label{lem:good-linear-forest}
			Let $\cmin \in (0,1)$ and $n \in \N$, and let $\eta, \beta, \xi,
			\gamma, \zeta, \delta$ be real numbers satisfying $1/n \ll \eta \ll
			\beta \ll \xi \ll \gamma \ll \zeta \ll \delta \ll \cmin$.  Let $G$
			be a $d$-regular graph on $n$ vertices, where $d \ge \cmin n$, and
			let $\{A_1, \ldots, A_r\}$ be a partition of $V(G)$ with 
			properties~\ref{itm:few-cross-edges} to \ref{itm:close-far-bip} in
			\Cref{lem:good-partition}, where $r \le \ceil{1 / \cmin}$. For each $i \in [r]$ such that $A_i$ is
			\nearb{\beta}, let $\{X_i, Y_i\}$ be a partition of $A_i$ that
			maximises the number of \crossedges{X_i}{Y_i} edges.  Then there
			is a linear forest $H \subset G$ with the following properties: 
			\begin{enumerate}
				\item \label{itm:H-small}
					$|H| \le \xi n$;
				\item \label{itm:H-no-isolated}
					$H$ has no isolated vertices;
				\item \label{itm:H-two-ends}
					for each $i \in [r]$, $A_i$ contains either two or zero leaves of $H$;
				\item \label{itm:H-ends-equally}
					for each $i \in [r]$ such that $A_i$ is \nearb{\beta},
					either $A_i$ contains no leaves of $H$, or $X_i$ and
					$Y_i$ each contain exactly one leaf of $H$;
				\item \label{itm:H-balances}
					for each $i \in [r]$ such that $A_i$ is \nearb{\beta},
					$|X_i \sm V(H)| = |Y_i \sm V(H)|$.
			\end{enumerate}
		\end{lem}
	\subsection{Proof of the main result}

		We now complete the proof of \Cref{thm:main}, using
		\Cref{lem:good-partition,lem:rob-ham,lem:good-linear-forest}. The
		proof mostly puts the three lemmas together, but we need to work a
		bit to get the exact right number of cycles. The lemmas themselves
		will be proved in forthcoming sections.

		\begin{proof}[ of \Cref{thm:main}]
			Let $\cmin > 0$; we assume, without loss of generality, that $1 /
			\cmin \in \N$. Let $n_0 \in \N$ satisfy $1 / n_0 \ll \cmin$, and
			let $G$ be a $d$-regular graph on $n$ vertices, where $n \ge n_0$
			and $d \ge \cmin n$.  Let $\{A_1, \ldots, A_r\}$ be a partition of
			$V(G)$ produced by \Cref{lem:good-partition}; this partition comes
			with parameters $1/n_0 \ll \eta \ll \beta \ll \gamma \ll \zeta \ll
			\delta \ll \cmin$. Set  $l = \floor{n / (d + 1)}$ and let $\alpha$
			be such that $\delta \ll \alpha \ll \cmin$.  
			
			For the moment, we fix a single index $i \in [r]$. By
			property~\ref{itm:min-deg} in \Cref{lem:good-partition}, $|A_i| \ge
			\delta n$. Hence, by property~\ref{itm:few-cross-edges}, there
			exists a vertex $u \in A_i$ incident with at most $(\eta / \delta)
			n$ edges of $G$ that leave $A_i$. Therefore, $u$ has at least $d -
			(\eta / \delta) n$ neighbours in $A_i$, and so $|A_i| \ge d - (\eta
			/ \delta) n \ge d \left( 1 - \eta / (\delta \cmin) \right) \ge (1 -
			\alpha) d$ (using $d \ge \cmin n$ and $\eta \ll \beta \ll \alpha
			\ll \cmin$). More can be said if $A_i$ is \nearb{\beta}. In such
			case we fix a partition $\{X_i, Y_i\}$ of $A_i$ that maximises the
			number of \crossedges{X_i}{Y_i} edges in $G$. In particular,
			$G[X_i, Y_i]$ can be obtained from $G[A_i]$ by removing at most
			$\beta n^2$ edges. Similarly to the argument above, there exists a
			vertex in $A_i$, say in $X_i$, with at least $d - \left((\eta +
			2\beta) / \delta \right) n$ neighbours in $G[X_i, Y_i]$, which
			means that $|Y_i| \ge d(1 - (3\beta / \delta \cmin)) \ge (1 -
			\alpha)d$.  Therefore, some vertex in $Y_i$ has at least $d - (\eta
			+ 2\beta) n^2 / (1 - \alpha) d$ neighbours in $G[X_i, Y_i]$, which
			implies that $|X_i| \ge d \left( 1 - 3 \beta / (1 - \alpha) \cmin^2
			\right) \ge (1 - \alpha)d$. We conclude that $|A_i| \ge (1 -
			\alpha) d$ in general and $|A_i| \ge 2 (1 - \alpha) d$ if $A_i$ is
			\nearb{\beta}.
			
			Since $i \in [r]$ was arbitrary, we have
			$n \ge (r + s) (1 - \alpha) d$, where $s$ is the number of
			\nearb{\beta} $A_i$'s. It follows that $r + s \le l + 1$: this can
			be seen by bounding the difference
			\begin{equation*}
				r + s - l \le
				\left\lfloor \frac{n}{(1 - \alpha) d} - \left\lfloor \frac{n}{d + 1} \right\rfloor \right\rfloor \le
				\left\lfloor \frac{\alpha d n + n}{(d + 1) d (1 - \alpha)} + 1 \right\rfloor
					\le
				\left\lfloor \frac{\alpha + 1 / n}{(\cmin)^2 (1 - \alpha)} + 1 \right\rfloor
				= 1.
			\end{equation*}

			The rest of the proof splits into two cases: when $r \le l$ and
			when $r = l + 1$. We first deal with the former case, which is
			critical, but easy to resolve using \Cref{lem:good-linear-forest}.
			We fix an arbitrary number $\xi$ such that $\beta \ll \xi \ll
			\gamma$.  Let $H$ be a linear forest as produced by
			\Cref{lem:good-linear-forest} (for each \nearb{\beta} $A_i$ we use
			the partition $\{X_i, Y_i\}$ that was defined earlier in the
			proof), and we denote by $I$ the set of internal vertices of $H$.
			For each $i \in [r]$, if $A_i$ contains two leaves of $H$, then let
			$x_i, y_i$ be those leaves. Otherwise, let $x_i, y_i \in A_i \sm I$
			be arbitrary adjacent vertices. Recall that $|I| \le \xi n$ by
			property~\ref{itm:H-small} in \Cref{lem:good-linear-forest}. We
			make two further observations if $A_i$ is \nearb{\beta}. First,
			property~\ref{itm:H-ends-equally} in \Cref{lem:good-linear-forest}
			enables us to assume that $x_i \in X_i$ and $y_i \in Y_i$. Second,
			properties~\ref{itm:H-ends-equally} and \ref{itm:H-balances} in
			\Cref{lem:good-linear-forest} imply that $|X_i \sm I| = |Y_i \sm
			I|$. Now, we apply \Cref{lem:rob-ham} and conclude that, regardless
			of $A_i$ being \nearb{\beta} or \farb{\gamma}, $G[A_i \sm I]$ has a
			Hamilton path with ends $x_i, y_i$. We take these paths for all $i
			\in [r]$: some of them can be concatenated with the path components
			of $H$, while the others have adjacent ends and so can be completed
			into cycles.  The result is a family of cycles that partitions
			$V(G)$. Note that the number of cycles in this family does not
			exceed the number of clusters, which is $r \le l$.
		
			We move on to the next case, that is, when $r = l + 1$. This
			immediately implies that $s = 0$, meaning that all $A_i$'s are
			\farb{\gamma}. Suppose that there is a matching of size $2$ between
			two distinct clusters $A_i, A_j$, and denote its edges by $x_i x_j$
			and $y_i y_j$, where $x_i, y_i \in A_i$ and $x_j, y_j \in A_j$. By
			\Cref{lem:rob-ham}, for each $k \in \{i,j\}$ there is a Hamilton
			path in $G[A_k]$ with ends $x_k$ and $y_k$. Together with the edges
			$x_i y_i$ and $x_j y_j$, we obtain a cycle whose vertex set is $A_i
			\cup A_j$. For every $k \neq i, j$, we use \Cref{lem:rob-ham} again
			to find a cycle with vertex set $A_k$. In total we obtain a
			partition of the vertices into $l$ cycles.

			Now, let us assume for contradiction that there are no two distinct
			clusters with a matching of size $2$ between them (i.e.\ the non-isolated vertices of $G[A_i, A_j]$ form a star for every $i \neq j$). We construct an
			auxiliary digraph $H$ on vertices $V(G)$, whose arcs correspond to
			edges of $G$ that join separate clusters. More precisely, for any
			distinct $i, j \in [r]$ such that $G$ contains \crossedges{A_i}{A_j} we do the following. If there is exactly one \crossedges{A_i}{A_j} edge $xy$, we add both $xy$ and $yx$ to $H$. Otherwise, since $G[A_i, A_j]$ is a star with at least two edges, there is a unique vertex $a \in A_i \cup A_j$ such that all \crossedges{A_i}{A_j} edges are incident with $a$. Add to $H$ all \crossedges{A_i}{A_j} edges as arcs directed towards $a$.
			
			In order to complete the proof, we reach a contradiction by a double-counting argument. Intuitively, the structure of $H$ suggests that there are few edges with ends in distinct parts $A_i$, but the assumption that there are $r = l+1$ parts $A_i$ implies that there are relatively many such edges.
			Fix $i \in [r]$. Let $a_i$ denote the number of arcs in $H$ that
			enter $A_i$ and let $b_i$ denote the number of arcs that leave
			$A_i$. Our most immediate aim is to establish the inequality
			\begin{equation} \label{ineq:inarcs-vs-outarcs}
				b_i \ge (d - 2l) (d + 1 - |A_i|) + a_i - 2l^2.
			\end{equation}
			To this aim, we first observe that $|A_i| \ge d - l$, or else in $G$
			every vertex of $A_i$ would send at least $l + 1$ edges to the other
			clusters. By the pigeonhole principle, at least two of these edges
			would end in the same cluster, and hence, again by the pigeonhole
			principle, there would be a cluster $A_j$, $j \neq i$, such that at
			least $|A_i| / l \ge 2$ vertices in $A_i$ send at least two edges to
			$A_j$.  However, this would contradict the assumption that there is
			no matching of size $2$ between any two clusters. 
			Furthermore, for
			any $j \neq i$, all arcs of $H$ that go from $A_j$ to $A_i$ have the
			same head. Therefore, there are at least $d - 2l$ vertices in $A_i$
			of zero in-degree in $H$. We pick a set $Z \subset A_i$ consisting
			of exactly $d - 2l$ such vertices.

			We write $m$ for the number of \crossedges{(A_i \sm Z)}{Z} edges
			missing from $G$ and denote the number of vertices in $A_i$ of
			non-zero in-degree in $H$ by $k$. We already know that $k \le l$. In
			$G$, these vertices together send at least $a_i$ edges outside of
			$A_i$, and so they send at most $kd - a_i$ edges to $Z$.  Therefore,
			$m \ge k (d - 2l) - (kd - a_i) \ge a_i - 2 l^2$.  Since $\sum_{z \in
			Z} |N_G(z) \cap A_i| \le |Z| (|A_i| - 1) - m$, there are at least
			$|Z|(d + 1 - |A_i|) + m \ge (d - 2l)(d + 1 - |A_i|) + a_i - 2 l^2$
			edges from $Z$ to $V(G) \setminus A_i$. They become arcs of $H$ directed away from
			$A_i$, proving inequality~\eqref{ineq:inarcs-vs-outarcs}.

			Summing inequality~\eqref{ineq:inarcs-vs-outarcs} over $i \in [r]$,
			we get
			\begin{equation*}
				0 =
				\sum_{i = 1}^r (b_i - a_i) \ge
				(d - 2l) \left( (d + 1)r - n \right) - 2l^2 r.
			\end{equation*}
			Since $r = \floor{n / (d + 1)} + 1 > n / (d + 1)$, we have $(d + 1)r
			- n \ge 1$, and hence the right hand side of the inequality above is
			at least $\cmin n - 2l - 2l^2 (l + 1) > 0$, giving a
			contradiction.
		\end{proof}

		In the proof of our main theorem, which we have just completed, we
		partitioned $V(G)$ into at most $l = \floor{n / (d + 1)}$ cycles. This
		proof can be tweaked so that exactly $l$ cycles are guaranteed: if
		the original proof produces $l' < l$ cycles, then before invoking
		\Cref{lem:rob-ham} to find Hamilton paths in the clusters, we can
		first take aside $l - l'$ very short cycles in one of the clusters
		(short cycles exist in clusters by \Cref{prop:short-connection}).

		If, instead of cycles, we wanted to partition $V(G)$ into (at most)
		$l$ paths, then the analysis of the case $r = l + 1$ in the proof of
		\Cref{thm:main} would be simpler. Indeed, instead of finding a
		matching of size $2$ between two clusters it would be enough to find
		a single edge.

	\subsection{Proof of the bipartite analogue}

		We now prove \Cref{thm:main-bip}, which is the bipartite analogue of
		our main result. As long as we have
		\Cref{lem:good-partition,lem:rob-ham,lem:good-linear-forest} at our
		disposal, the proof is straightforward, but, again, some care is
		needed to obtain the exactly tight bound.

		\begin{proof} [ of \Cref{thm:main-bip}] 
			Let $\cmin$ be such that $1/\cmin \in \N$, and suppose that $1 / n_0 \ll \cmin$.
			Let $G$ be a bipartite $d$-regular graph on $n$ vertices, where $n \ge n_0$ and $d \ge \cmin n$.
			Let $X, Y$ be the vertex classes of $G$ and write $l = \floor{n
			/ (2d)}$. Let $\{A_1, \ldots, A_r\}$ be a partition of $V(G)$ as
			given by \Cref{lem:good-partition}, where $1/n_0 \ll \eta \ll
			\zeta \ll \delta \ll \cmin$ are the corresponding parameters ($\beta$ and
			$\gamma$ do not play a role here as the graph is bipartite). The
			argument that applied to \nearb{\beta} clusters in the proof of
			\Cref{thm:main} also works here and it shows that given $\alpha$ that satisfies $\delta \ll \alpha \ll \cmin$, we have $|A_i| \ge 2d(1 -
			\alpha)$ for all $i \in [r]$. Therefore,
			\begin{equation*}
				r - l \le
				\left \lfloor
					\frac{n}{2d(1 - \alpha)} - \frac{n}{2d} + 1
				\right \rfloor \le
				\left \lfloor
					\frac{\alpha}{2 \cmin (1 - \alpha)} + 1
				\right \rfloor =
				1.
			\end{equation*}
			
			Let $\xi$ be a parameter satisfying $\eta \ll \xi \ll \zeta$ and
			for each $i \in [r]$ fix the partition $\{X_i, Y_i\}$ for $A_i$,
			where $X_i = A_i \cap X$, $Y_i = A_i \cap Y$.  Let $H$ be a
			linear forest as given by \Cref{lem:good-linear-forest}.
			Precisely as in the proof of \Cref{thm:main}, by concatenating
			components of $H$ and paths in the clusters, we partition $V(G)$
			into at most $r$ cycles. Furthermore, if at least one component
			of $H$ has ends in separate clusters, then the partition
			contains at most $r - 1$ cycles. Therefore, we may assume that
			$r = l + 1$ and both ends of each component of $H$ are in the
			same cluster, as otherwise we are done.
			
			Now, suppose that $H$ has an edge $uv$ with ends in separate
			clusters, say, $u \in X_1$, $v \in Y_2$. Let $P$ be the
			component of $H$ that contains $uv$, and let $x_1, y_1$ be the
			ends of $P$ in $X_1, Y_1$, respectively (both parts of $A_1$
			contain an end of $P$ by property~\ref{itm:H-ends-equally} in
			\Cref{lem:good-linear-forest}). We write $P_u, P_v$ for the two
			paths comprising $P \sm \{uv\}$, where $P_u$ contains $u$ and
			$P_v$ contains $v$ ($P_u$ and/or $P_v$ is a single vertex if $u$
			and/or $v$ is an end of $P$). We select a vertex $x_2 \in X_2$
			in the following way: if $H$ has a component with ends in $A_2$,
			then we let $x_2$ be its end in $X_2$; otherwise, we pick $x_2$
			arbitrarily.  Note that $|(X_1 \sm V(H)) \cup \{x_1\}| = |(Y_1
			\sm V(H)) \cup \{y_1\}|$ by property~\ref{itm:H-balances} in
			\Cref{lem:good-linear-forest}, and hence \Cref{lem:rob-ham}
			produces a path with ends $x_1, y_1$ that spans $(A_1 \sm V(H))
			\cup \{x_1, y_1\}$. Similarly, there is a path spanning $(A_2
			\sm V(H)) \cup \{x_2, v\}$ that has ends $x_2, v$. Let $P^\ast$
			be the concatenation of $P_u$ with the newly produced path
			between $x_1, y_1$, with $P_v$, with the newly produced path
			between $v, x_2$ and, if it exists, with the component of $H$
			whose one end is $x_2$. We observe that $P^\ast$ is a path that
			covers $A_1 \cup A_2$ except for the vertices that appear in
			components of $H$ with ends in clusters other than $A_1, A_2$.
			Outside of $A_1 \cup A_2$, $P^\ast$ covers the vertices
			contained in components of $H$ with ends in $A_1, A_2$. We deal
			with the clusters $A_i$ for $i \ge 3$ in the same way as in the
			proof of \Cref{thm:main}.  This gives a partition of $V(G)$ into
			the path $P^\ast$ and at most $r - 2 = l - 1$ cycles, proving
			the result.

			The final case to consider is when $r = l + 1$ and $H$ has no
			edges with ends in separate clusters. By
			property~\ref{itm:H-ends-equally} in
			\Cref{lem:good-linear-forest}, every component of $H$ covers the
			same number of vertices in both parts of the graph. Therefore,
			for each $i \in [r]$, $|X_i| = |X_i \sm V(H)| + |X_i \cap V(H)|
			= |Y_i \sm V(H)| + |Y_i \cap V(H)| = |Y_i|$. In other words,
			each cluster is balanced. Since $r > n / (2d)$, we may assume
			that $|A_1| < 2d$, and so $|X_1| = |Y_1| < d$. By the regularity
			of $G$, there exists an edge $uv \in E(G)$ with $u \in X_1$ and
			$v$ not in $Y_1$. Say, $v \in Y_2$. By \Cref{lem:rob-ham}, for
			each $i \in [r]$ we may pick a path $P_i$ spanning $A_i$, where
			$u$ is an end of $P_1$ and $v$ is an end of $P_2$. This gives a
			partition of $V(G)$ into $r - 1 = l$ paths, namely, $P_1 uv
			P_2$, $P_3, \dotsc, P_r$.
		\end{proof}

		We remark that a possible strategy for proving a stronger version of
		\Cref{thm:main-bip} that establishes a partition of $V(G)$ into at
		most $\floor{n / (2d)}$ cycles may revolve around moving a small
		number of vertices from some clusters to others, so that the
		clusters still satisfy the assumptions of \Cref{lem:rob-ham}, but
		the balancing linear forest now has a component with ends in
		separate clusters. We believe that we have a good idea on how such a
		proof would work -- a more technical version of
		\Cref{lem:good-linear-forest} is needed -- but we decided not to
		pursue it.
		
	\subsection{Structure of the paper}
	
	We prove \Cref{lem:good-partition,lem:rob-ham,lem:good-linear-forest} in \Cref{sec:clusters,sec:ham,sec:balancing}, respectively, and conclude the paper in \Cref{sec:conclusion} with closing
		remarks and open problems.

\section{Partitioning the graph into well-behaved clusters} \label{sec:clusters}

	\def \rmax{r_0}

	In this section we prove \Cref{lem:good-partition}. 
	This lemma is very similar to Theorem 3.1 in \cite{kuhn-lo-osthus-staden-approx}\footnote{The main conceptual difference is that we prove that each set $A_i$ has no sparse cuts, whereas in \cite{kuhn-lo-osthus-staden-approx} it is proved that each $G[A_i]$ is a robust expander; in fact, the latter would work for us as well, but we chose the former to simplify the presentation.}. Nevertheless, as the proof in \cite{kuhn-lo-osthus-staden-approx} is quite long, we give a proof here.

	\begin{proof}[ of \Cref{lem:good-partition}]
		Set $\rmax := \ceil{1/\cmin}$ and fix positive constants $n_0$ and $\eta_1, \dotsc,
		\eta_{\rmax}$ that satisfy the hierarchy $1 / n_0 \ll \eta_1
		\ll \dotsb \ll \eta_{r_0} \ll \cmin$. Let $G$ be a $cn$-regular
		graph on $n \ge n_0$ vertices, where $c \ge \cmin$. We shall define a
		list $\P_1, \dotsc, \P_r$, where $1 \le r \le \rmax$, of
		increasingly refined partitions of $V(G)$ such that the following
		properties hold for each $i \in [r]$:
		\begin{enumerate}[label = \rm{(\roman*)}]
			\item \label{itm:partn-one}
				$\P_i$ is a partition of $V(G)$ consisting of $i$ non-empty
				parts;
			\item \label{itm:partn-two}
				if $i \ge 2$, then $\P_i$ is obtained by splitting one part of
				$\P_{i-1}$ into two;
			\item \label{itm:partn-three}
				$G$ has at most  $4 \sqrt{\eta_{i-1}} n^2$ edges with ends in
				different parts of $\P_i$ (where $\eta_0 = 0$ by convention);
			\item \label{itm:partn-four}
				for every $A \in \P_i$, the minimum degree of $G[A]$ is at least
				$3^{-(i-1)} cn$;
			\item \label{itm:partn-terminates}
				every part of $\P_r$ has no $\eta_r$-sparse cuts.
		\end{enumerate}
			Let $\P_1 = \{V(G)\}$ and note that $\P_1$ trivially satisfies the
			first four conditions. Assuming that $\P_i$ is defined,
			we define $\P_{i+1}$ in the following way. If every part of $\P_i$
			has no $\eta_i$-sparse cut, then we set $r=i$ and stop the process.
			Otherwise, we pick a part $A \in \P_i$ that has an $\eta_i$-sparse
			cut $\{A_1, A_2\}$.  In $A_1$, we let $A_1'$ be the set of vertices
			that have at most $\sqrt{\eta_i} n$ neighbours in $A_2$; similarly,
			we denote by $A_2'$ the set of vertices in $A_2$ that have at most
			$\sqrt{\eta_i} n$ neighbours in $A_1$. Since $\{A_1, A_2\}$ is an
			$\eta_i$-cut of $A$, we have $\sqrt{\eta_i} n |A \sm (A_1' \cup
			A_2')| \le 2\eta_i n^2$, and hence $|A \sm (A_1' \cup A_2')| \le
			2\sqrt{\eta_i} n$.  Since every vertex in $A$ has at least
			$3^{-(i-1)} cn$ neighbours in $A$ and since all but at most $2
			\sqrt{\eta_i} n < 3^{-i} cn$ of them are in $A_1' \cup A_2'$, every
			vertex in $A$ has at least $3^{-i} cn$ neighbours in $A_j'$ for some
			$j \in \{1,2\}$. In particular, $G[A_1']$ and $G[A_2']$ both have
			minimum degree at least $3^{-i} cn$. Furthermore, we can partition
			$A \sm (A_1' \cup A_2')$ into sets $A_1'', A_2''$ where for each $j
			\in \{1, 2\}$ every vertex in $A_j''$ has at least $3^{-i} cn$
			neighbours in $A_j'$. We define $\P_{i+1}$ by replacing the part $A$
			in $\P_i$ with two parts $A_1' \cup A_1''$ and $A_2' \cup A_2''$. It is
			clear that $\P_{i+1}$ satisfies properties~\ref{itm:partn-one},
			\ref{itm:partn-two} and \ref{itm:partn-four}.

			We now prove that $\P_{i+1}$ satisfies property~%
			\ref{itm:partn-three}, provided that $i \le \rmax$ (we
			will show in the next paragraph that the process in fact terminates
			at some $\P_r$ with $r \le \rmax$).
			The number of edges between $A_1' \cup A_1''$ and $A_2' \cup
			A_2''$ is at most $\eta_i n^2 + |A_1''\cup A_2''| cn \le
			(\eta_i + 2 \sqrt{\eta_i}) n^2 \le 3 \sqrt{\eta_i} n^2$. Hence, by
			property~\ref{itm:partn-three} of $\P_i$ and by the assumption
			that $\eta_{i-1} \ll \eta_i$, the number of edges between the parts
			of $\P_{i+1}$ is at most $(4 \sqrt{\eta_{i-1}} + 3\sqrt{\eta_i})n^2
			\le 4 \sqrt{\eta_i} n^2$, as desired.

			If the process does not terminate for any $i \le \rmax$,
			then we create a partition $\P_{\rmax + 1}$ that satisfies
			properties~\ref{itm:partn-one} to \ref{itm:partn-four}. We will
			show that such a partition is impossible. Let $A$ be a part of
			$\P_{\rmax + 1}$ of the least order. Clearly, $|A| \le n
			/ ((1/c) + 1) = cn/(c+1)$, and so every vertex in $A$ has at least $cn
			( 1 - 1/(c+1)) = c^2 n / (c+1) \ge \cmin^2 n / 2$ neighbours outside
			of $A$.  Moreover, property~\ref{itm:partn-four} implies that $|A|
			\ge 3^{-\rmax} \cdot \cmin n$. Therefore,
			property~\ref{itm:partn-three} implies that
			\begin{equation*}
				3 \sqrt{\eta_{\rmax}} n^2
				\;\ge\; |A| \cdot (\cmin^2 \, n/2)
				\;\ge\; \frac{1}{2}\, 3^{-\rmax} \, \cmin^2 \, n^2,
			\end{equation*}
			contradicting the assumption that $\eta_{\rmax} \ll \cmin$.

			Consider the final partition $\P_r$. It consists of $r \le \rmax$ parts, none of which have $\eta_r$-sparse cuts.  We set
			$\zeta = \eta_r,\, \eta = 3 \sqrt{\eta_{r-1}},\, \delta = 3^{-r} c$
			and observe that $\P_r$ satisfies
			properties~\ref{itm:few-cross-edges} to \ref{itm:no-small-cuts} in
			\Cref{lem:good-partition}. For property~\ref{itm:close-far-bip},
			we fix positive coefficients $\beta_0, \dotsc, \beta_{r+1}$ that
			depend only on $\cmin$ and $r$, satisfying $3\sqrt{\eta_{r-1}} =
			\eta \ll \beta_0 \ll \dotsb \ll \beta_{r+1} \ll \zeta = \eta_r$.
			For $i \in \{0, \ldots, r+1\}$, let $b(i)$ be the number of parts
			$A$ in $\P_r$ that are \nearb{\beta_i}. Note that if $A$ is
			\nearb{\beta_i} it is also \nearb{\beta_{i+1}}. In particular, $0 \le
			b(0) \le \ldots \le b(r+1) \le r$. It follows that there exists $i
			\in \{0, \ldots, r\}$ with $b(i) = b(i+1)$; fix one such $i$. Then
			every part $A$ in $\P_r$ is either \nearb{\beta_i} or
			\farb{\beta_{i+1}}. Therefore, we can finish the proof by setting
			$\beta = \beta_i$ and $\gamma = \beta_{i+1}$.
		\end{proof}

\section{Hamiltonicity of clusters} \label{sec:ham}

	In this section we prove \Cref{lem:rob-ham}.
	Our proof relies on known results\footnote{In \href{https://https://arxiv.org/abs/1808.00851v1}{\textcolor{blue}{arXiv1808.00851v1}} we prove \Cref{lem:rob-ham} from scratch.} regarding the Hamiltonicity of so-called
	robust out-expanders, a notion that was introduced by K\"uhn, Osthus and
	Treglown \cite{kuhn-osthus-treglown}. Before mentioning the relevant result, we make some definitions. 
	
	Given a digraph $G$
	on $n$ vertices, a set of vertices $S$ and a parameter $\nu \in (0, 1)$,
	the \emph{robust $\nu$-out-neighbourhood} of $S$ in $G$, denoted
	$\RN^+_{\nu, G}(S)$, is the set of vertices in $G$ that have at least $\nu
	n$ in-neighbours in $S$; we omit the subscript $G$ when it is clear from
	the context.  Given $0 < \nu \le \tau < 1$, we say that $G$ is a
	\emph{robust $(\nu, \tau)$-out-expander} if $|\RN^+_{\nu}(S)| \ge |S| + \nu n$
	for every set of vertices $S$ with $\tau n \le |S| \le (1 - \tau) n$.  We
	shall also use the undirected version of a robust out-neighbourhood: in a
	graph $G$ on $n$ vertices, the \emph{robust $\nu$-neighbourhood} of a set
	of vertices $S$, denoted $\RN_{\nu, G}(S)$ is the set of vertices in $G$
	with at least $\nu n$ neighbours in $S$; as before we sometimes omit the
	subscript $G$.

	We shall use the following theorem from \cite{kuhn-osthus-treglown}; recall
	that $\delta^0(G) = \min\{ \delta^+(G), \delta^-(G)\}$, where $\delta^+(G),
	\delta^-(G)$ are the minimum out-degree and in-degree of $G$, respectively.
	\begin{thm} \label{thm:robust-expanders-hamilton}
		Let $n_0 \in \N$ and let $\gamma, \nu, \tau$ be reals such that $1 /
		n_0 \ll \nu \le \tau \ll \gamma < 1$. Let $G$ be a digraph on $n \ge
		n_0$ vertices with $\delta^0(G) \ge \gamma n$ which is a robust $(\nu,
		\tau)$-out-expander. Then $G$ contains a Hamilton cycle.
	\end{thm}

	In fact, we shall need the following corollary.
	\begin{cor} \label{cor:robust-expanders-hamilton}
		Let $n_0 \in \N$ and let $\gamma, \nu, \tau$ be reals such that $1 /
		n_0 \ll \nu \le \tau \ll \gamma < 1$. Let $G$ be a digraph on $n \ge
		n_0$ vertices with $\delta^0(G) \ge \gamma n$ which is a robust $(\nu,
		\tau)$-out-expander. Then for every choice of distinct vertices $x, y$, there is a Hamilton path in $G$ with ends $x, y$.
	\end{cor}

	\begin{proof}
		Given vertices $x, y$, form $G'$ by adding the arc $xy$ to $G$, removing the arc $yx$ (if it exists), and removing all edges directed towards $y$ or from $x$.
		Next, form $G''$ by contracting the arc $xy$. It is easy to check that $G''$ is a robust $(\nu/2, 2\tau)$-out-expander. Thus, by \Cref{thm:robust-expanders-hamilton}, it contains a Hamilton cycle. This cycle corresponds to a Hamilton cycle in $G'$ which contains the arc $xy$, which in turn corresponds to a Hamilton path in $G$ with ends $x,y$.
	\end{proof}

	\begin{proof} [ of \Cref{lem:rob-ham}]
		Let $A$ satisfy properties \ref{itm:few-cross-edges-A} to
		\ref{itm:close-far-bip-A} in \Cref{lem:rob-ham}; if $A$ is
		\nearb{\beta}, let $\{X, Y\}$ be a partition of $A$ that maximises the
		number of \crossedges{X}{Y} edges. Let $W \subset A$ be a set of size at
		most $\eta n$; if $A$ is \nearb{\beta} we further assume that $|X
		\setminus W| = |Y \setminus W|$. Let $H$ be the subgraph of $G$ defined
		as follows: if $A$ is \farb{\gamma} set $H = G[A']$, and otherwise set
		$H = G[X', Y']$, where $A' = A \setminus W$, $X' = X \setminus W$ and
		$Y' = Y \setminus W$. 	

		The following claim will allow us to use \Cref{cor:robust-expanders-hamilton} above; its proof is somewhat technical. 

		\begin{claim} \label{claim:robust-nbd}
			Let $S \subset A'$ be a set satisfying $\xi^{1/7} |A'| \le |S| \le (1 - \xi^{1/7})|A'|$ if $A$ is \farb{\gamma}, or $\xi^{1/7} |A'| \le |S| \le (1/2 - \xi^{1/7})|A'|$ if $A$ is \nearb{\beta}.
			Then $\RN_{\xi, H}(S) \ge |S| + \xi n$.
		\end{claim}	

		\begin{proof}
			We define $S_1 = S \setminus \RN_{\xi}(S)$, $S_2 = S \cap \RN_{\xi}(S)$, $T_1 = \RN_{\xi}(S) \setminus S$, $T_2 = A' \setminus (S \cup T_1)$. We assume that $|\RN_{\xi}(S)| < |S| + \xi n$, which implies that $|T_1| < |S_1| + \xi n$.
			Write $V = V(G)$. Given sets $X, Y \subset V$, let $e(X, Y)$ be the number of ordered pairs $xy$ such that $xy$ is an edge of $G$ and $x \in X, y \in Y$.
			\begin{equation} \label{eqn:robust-a}
				e(S_1, V \setminus T_1) \le e(A, V \setminus A) + e(W, V \setminus W) + e(S_1, S) + e(S_1, T_2) \le 5 \xi n^2, 
			\end{equation}
			where we used property~\ref{itm:few-cross-edges-A} in \Cref{lem:rob-ham}, the assumption that $|W| \le \xi n$, the fact that vertices in $S_1 \cup T_2$ are not in $\RN_{\xi}(S)$, and the fact that $H$ is obtained from $G[A']$ by removing at most $\beta n^2$ edges.
			It follows that $e(S_1, T_1) \ge |S_1| d - 5 \xi n^2$. As $|T_1| \le |S_1| + \xi n$, we obtain the following bound.
			\begin{equation} \label{eqn:robust-b}
				e(T_1, V \setminus S_1) \le |T_1| d  - e(S_1, T_1) 
				\le (|T_1| - |S_1|)d + 5 \xi n^2 
				\le 6 \xi n^2.
			\end{equation}
			Consider the quantity $e(S_1 \cup T_1, A \setminus (S_1 \cup T_1))$. 
			By \eqref{eqn:robust-a} and \eqref{eqn:robust-b}, it is at most $11 \xi n^2$, and by property \ref{itm:no-small-cuts-A} in \Cref{lem:rob-ham}, it is at least $\zeta |S_1 \cup T_1| (|A \setminus (S_1 \cup T_1)|)$.
			As $\xi \ll \zeta$, we find that either $|S_1 \cup T_1| \le \xi^{1/3} n$ or $|A \setminus (S_1 \cup T_1)| \le \xi^{1/3} n$.

			Suppose first that $|S_1 \cup T_1| \le \xi^{1/3} n$. Then
			\[
				e(S_2, A \setminus S_2) \le e(W, V) + e(S_1 \cup T_1, V) + e(S_2, T_2) \le 2\xi^{1/3} n^2,
			\]
			using $|W| \le \xi n$, $|S_1 \cup T_1| \le \xi^{1/3} n$, and $T_2 \cap \RN_{\xi}(S_2) = \emptyset$.
			But, by the assumptions of \Cref{claim:robust-nbd}, $|S_2| \ge \xi^{1/7} |A'| - |S_1| \ge (\xi^{1/7} / 2)n$, and $|A' \setminus S_2| \ge \xi^{1/7}|A|$, thus by property~\ref{itm:no-small-cuts-A} in \Cref{lem:rob-ham} we have $e(S_2, A \setminus S_2) \ge \zeta |S_2| \cdot |A \setminus S_2| > 2 \xi^{1/3} n^2$, a contradiction.

			Next, suppose that $|A \setminus (S_1 \cup T_1)| \le \xi^{1/3} n$. If $A$ is \nearb{\beta} then $|S_1 \cup T_1| \le 2|S_1| + \xi n \le (1 - 2 \xi^{1/7}) |A'| + \xi n < |A| - \xi^{1/3} n$, a contradiction. So $A$ is \farb{\gamma}. Note that $G[A]$ can be made bipartite by removing edges incident with $W \cup S_2 \cup T_2$ or within $S_1$ or $T_1$. But there are at most $(\eta + \xi^{1/3})n^2$ edges of the former type, and at most $11 \xi n^2$ edges of the latter type (by \eqref{eqn:robust-a} and \eqref{eqn:robust-b}), so fewer than $\gamma n^2$ edges in total (using $\xi, \eta \ll \gamma$). This is a contradiction to the fact that $A$ is \farb{\gamma}, completing the proof.
		\end{proof}

		Let $x, y \in A'$, where $x \in X', y \in Y'$ if $A$ is \nearb{\beta}.
		Out task is to show that $H$ contains a Hamilton path with ends $x$ and
		$y$. First, we consider the case where $A$ is \farb{\gamma}. Form a digraph $D$ by replacing each edge $uv$ of $G$ by the two arcs $uv$ and $vu$. It follows from \Cref{claim:robust-nbd} that $D$ is a robust $(\xi, \xi^{1/7})$-out-expander. \Cref{cor:robust-expanders-hamilton} implies the existence of a Hamilton path with ends $x,y$, which corresponds to a Hamilton path in $G$ with the same ends.

		Now, suppose that $A$ is \nearb{\beta}. We claim that $H$ has a perfect matching. To this end, let $S \subset X'$; we show that $|N_H(S)| \ge |S|$. 
		Since $\delta(G[A]) \ge \delta n$, we have $\delta(G[X, Y]) \ge (\delta / 2) n$, because $X, Y$ were chosen to maximise the number of \crossedges{X}{Y} edges. It follows that $\delta(H) \ge (\delta/2 - \xi)n \ge (\delta / 3)n$.
		Thus, if $|S| \le (\delta / 3)n$, then, trivially, $|N_H(S)| \ge |S|$. Similarly, if $|S| > |X'| - (\delta / 3) n$, then every vertex in $Y'$ has a neighbour in $S$, and the desired inequality again follows. The remaining case is when $(\delta / 3) n \le |S| \le |X'| - (\delta / 3) n$, where the inequality $|N_H(S)| \ge |S|$ follows from \Cref{claim:robust-nbd}. 

		Let $\{a_1 b_1, \ldots, a_t b_t\}$ be a perfect matching in $H$, where $t = |X'|$ and $a_i \in X', b_i \in Y'$ for $i \in [t]$. We assume for convenience that $a_i b_i$ is not the edge $xy$ (if the latter exists) for $i \in [t]$ -- this is possible as the removal of the edge $xy$ from $H$ does not change the arguments above. Without loss of generality, $a_1 = x$ and $b_t = y$.
		Form a directed graph $D$ with vertex set $\{v_1, \ldots, v_t\}$ where $v_i v_j$ is an arc whenever $b_i a_j$ is an edge of $H$. It follows from \Cref{claim:robust-nbd} that $D$ is a robust $(2\xi, 2\xi^{1/7})$-out-expander, thus by \Cref{cor:robust-expanders-hamilton} there is a Hamilton path in $D$ with ends $v_1, v_t$. Without loss of generality, this path is $(v_1 \ldots v_t)$. This path corresponds to the Hamilton path $(x = a_1 b_1 \ldots a_t b_t = y)$ in $H$. 
	\end{proof}
\section{Balancing the bipartite clusters} \label{sec:balancing}

	In this section we prove \Cref{lem:good-linear-forest}.
	The proof spans the whole section and consists of several claims.

	\subsection{The setup}

		We first recap the setup needed for the proof of \Cref{lem:good-linear-forest}. 
		We are given parameters $\cmin, n, \eta, \beta, \xi, \gamma, \zeta, \delta$ such that
		\[
			1/n \ll \eta \ll \beta \ll \xi \ll \gamma \ll \zeta \ll \delta \ll \cmin.
		\]
		We are also given a $d$-regular graph $G$, where $d \ge \cmin n$, and we denote $d = cn$, so that $c \ge \cmin$. 
		We are further given a partition $\{A_1, \ldots, A_r\}$ of $V(G)$, where $r \le \ceil{1/\cmin}$, that satisfies the following properties. 
		\begin{enumerate}
			\item
				$G$ has at most $\eta n^2$ edges with ends in separate clusters;
			\item
				for each $i \in [r]$, the minimum degree of $G[A_i]$ is at least
				$\delta n$;
			\item
				for each $i \in [r]$, $A_i$ has no $\zeta$-sparse cuts;
			\item
				for each $i \in [r]$, either $A_i$ is \nearb{\beta}, in which case we fix $\{X_i, Y_i\}$ to be a partition of $A_i$ that maximises the number of \crossedges{X_i}{Y_i} edges, or $A_i$ is \farb{\gamma}.
		\end{enumerate}

		For the sake of the proof of \Cref{lem:balancing-matching}, $\xi$ denotes any parameter satisfying $\beta \ll \xi \ll \gamma$; we will not use the fact that each $A_i$ is $\xi$-Hamiltonian if it is \farb{\gamma}, or $\xi$-weakly-Hamiltonian if it is \nearb{\beta}, which follows from \Cref{lem:rob-ham}.

		Our aim is to find a linear forest $H$ in $G$, with the following properties.
		\begin{enumerate}
			\item \label{it:H-small}
				$|H| \le \xi n$;
			\item \label{it:H-no-isolated}
				$H$ has no isolated vertices;
			\item \label{it:H-two-ends}
				for each $i \in [r]$, $A_i$ contains either two or zero leaves of $H$;
			\item \label{it:H-ends-equally}
				for each $i \in [r]$ such that $A_i$ is \nearb{\beta}, either
				$A_i$ contains no leaves of $H$, or $X_i$ and $Y_i$ each contain
				exactly one leaf of $H$;
			\item \label{it:H-balances}
				for each $i \in [r]$ such that $A_i$ is \nearb{\beta}, $|X_i \sm
				V(H)| = |Y_i \sm V(H)|$.
		\end{enumerate}

		In the proof, we shall consider the \emph{lift}
		of $G$, denoted $\liftG$, which is a bipartite analogue of $G$.  The lift
		$\liftG$ is defined as follows. We set $V(\liftG) = \partb{V} \cup
		\partt{V}$ where $\partb{V}, \partt{V}$ are disjoint copies of $V(G)$; for
		every $i \in \{1,2\}$ and $v \in V(G)$ we denote by $v^{(i)}$ the copy of
		$v$ in $V^{(i)}$. For all $u,v \in V(G)$, $\partb{u}\partt{v}$ is an edge of
		$\liftG$ if and only if $uv$ is an edge of $G$. There are no edges in
		$\liftG$ with both ends in $\partb{V}$ or in $\partt{V}$. It is clear from
		this construction that $\liftG$ is a $cn$-regular bipartite graph on $2n$
		vertices. 
		
		The use of the lift $\liftG$ of $G$ is convenient for us for three reasons. First, we shall be dealing with flows and matchings, and the fact that $\liftG$ is bipartite makes it easier to analyse them. Second, the lift allows us to treat \nearb{\beta} and \farb{\gamma} clusters in a unified way. And, third, consider a \nearb{\beta} cluster $A_i$, with prescribed partition $\{X_i, Y_i\}$, and suppose that $|Y_i| = |X_i| + k$. It turns out that in order to `balance' $A_i$, it suffices to find two matchings $M_1, M_2$, whose union does not span any cycles or double edges, and, for $j \in \{1, 2\}$, we have $|V(M_j) \cap Y| = |V(M_j) \cap X|$. Because $\liftG$ contains two copies of each such cluster, a so-called balancing matching for $\liftG$ (we make the notion precise below), pulled back to $G$, provides us with such `overbalancing' automatically. In particular, if $G$ is bipartite, then $\liftG$ consists of two copies of $G$, and its analysis allows us to find two such matchings simultaneously.

		We partition the vertices of $\liftG$ into sets $\liftA_1, \dotsc,
		\liftA_s$, which we call \emph{clumps} (which are related to, but should not
		to be confused with clusters $A_1, \dotsc, A_r$), as follows.  Let
		$i$ be the index of an arbitrary \nearb{\beta} cluster $A_i$ of $G$ and fix
		a partition $\{X_i, Y_i\}$ of $A_i$ which maximises the number of
		\crossedges{X_i}{Y_i} edges in $G$.  In particular, $X_i, Y_i \neq \es$ and
		all but at most $\beta n^2$ edges of $G[A_i]$ are between $X_i$ and $Y_i$.
		Furthermore, every vertex of $X_i$ (resp.\  $Y_i$) has at least $\delta n/2$
		neighbours in $Y_i$ (resp.\ $X_i$), as otherwise we could move that vertex to
		the other part, increasing the number of \crossedges{X_i}{Y_i} edges. For $j
		\in \{1,2\}$, let $X_i^{(j)}, Y_i^{(j)}$ be the copies of, respectively,
		$X_i, Y_i$ in $V^{(j)}$. We define sets
		\begin{align*}
			\liftA_{i,1} &= B_{i,1} \cup T_{i,1},
			\text{ where } B_{i,1} = \partb{X_i} \text{ and } T_{i,1} = \partt{Y_i}, \\
			\liftA_{i,2} &= B_{i,2} \cup T_{i,2},
			\text{ where } B_{i,2} = \partb{Y_i} \text{ and } T_{i,2} = \partt{X_i}.
		\end{align*}
		Now, let $i$ be the index of some \farb{\gamma} cluster $A_i$. We
		define $B_i$ and $T_i$ to be the copies of $A_i$ in $\partb{V}$ and
		$\partt{V}$, respectively, and
		\begin{equation*}
			\liftA_i = B_i \cup T_i.
		\end{equation*}
		In these definitions $B$ stands for the `bottom part' and $T$ stands for the
		`top part'.

		By doing this for all $i \in [r]$ we obtain a partition of $V(\liftG)$ into
		clumps labelled $\liftA_{i,1}, \liftA_{i,2}$ (for those $i$ for which $A_i$
		is \nearb{\beta}) and $\liftA_i$ (for the other $i$). To make the notation
		consistent, we relabel these clumps simply as $\liftA_1, \dotsc, \liftA_s$,
		where $s = r + |\{ i \in [r] :A_i \text{ is \nearb{\beta}} \}|$.  In
		particular, $s \in \{r, \dotsc, 2r\}$. We relabel the sets $B_{\dotsc}$ and
		$T_{\dotsc}$ appropriately, so that $\liftA_j = B_j \cup T_j$ for all $j \in
		[s]$.

		\begin{obs} \label{obs:few-crossedges}
			$\liftG$ has at most $3 r \beta n^2$ edges with ends in separate clumps.
		\end{obs}

		\begin{proof}
			First, note that every edge with both ends in a \farb{\gamma} cluster
			$A_i$ of $G$ gives rise to two edges of $\liftG$, both contained in the
			clump corresponding to $A_i$.  Now, consider an arbitrary \nearb{\beta}
			cluster $A_j$ of $G$. We recall that $A_j$ is partitioned into sets
			$X_j, Y_j$ such that all but at most $\beta n^2$ edges of $G[A_j]$ are
			\crossedges{X_j}{Y_j} edges. In $\liftG$, $A_j$ gives rise to two
			clumps, say, $\liftA_{j_1}$ and $\liftA_{j_2}$. If $e \in E(G[A_j])$ is
			an \crossedges{X_j}{Y_j} edge, then $e$ corresponds to two edges of
			$\liftG$, one in $\liftA_{j_1}$ and one in $\liftA_{j_2}$.  Therefore,
			only those edges of $G[A_j]$ that are not \crossedges{X_j}{Y_j} edges
			give rise to edges of $\liftG$ with ends in separate clumps. Also, we
			have to account for the edges of $G$ that have ends in separate
			clusters. Thus, the number of edges of $\liftG$ with ends in separate clumps is at most $2\eta n^2 + 2r\beta n^2 \le \,3 r \beta n^2$,
			using the assumption that $\eta
			\ll \beta$.
		\end{proof}

		\begin{obs} \label{obs:cluster-min-degree}
			For each $i \in [s]$, the minimum degree of $\liftG[\liftA_i]$ is at
			least $\delta n / 2$. In particular, every vertex in $\liftA_i$ has at
			most $(c - \delta/2)n$ neighbours in $V(\liftG) \sm \liftA_i$.
		\end{obs}

		\begin{proof}
			Pick $i$ and let $A_j$ be the cluster of $G$ that gives rise to
			$\liftA_i$.  Let $v^{(t)}$ be an arbitrary vertex in $\liftA_i$, where
			$v \in A_j$, $t \in \{1,2\}$. If $A_j$ is \farb{\gamma}, then $v$ has at
			least $\delta n$ neighbours in $A_j$, and every such neighbour $u$ gives
			rise to the vertex $u^{(3-t)} \in \liftA_i$, which is adjacent to
			$v^{(t)}$.

			So suppose that $A_j$ is \nearb{\beta} with partition $A_j = X_j \cup
			Y_j$.  We recall that this partition was chosen so that every vertex in
			$X_j$ has at least $\delta n / 2$ neighbours in $Y_j$ and vice versa.
			Therefore, the number of
			\crossedges{X_j}{Y_j} edges incident with $v$ is at most $\delta n / 2$, and, for every such edge $uv$, the vertex
			$u^{(3-t)}$ is a neighbour of $v^{(t)}$ in $\liftG[\liftA_i]$.
			
			This proves the first part of the observation. Together with the fact
			that $\liftG$ is $cn$-regular, it implies the second part as well.
		\end{proof}

		Let $H$ be a bipartite graph with bipartition $\{X, Y\}$ and let $U \subset
		V(H)$. We define
		\begin{equation*}
			\disb_H(U) = \big| |U \cap X| - |U \cap Y| \big|.
		\end{equation*}
		We call this quantity the \emph{imbalance} of $U$ in $H$. If $H$ is clear
		from the context, then we may write $\disb(U)$ instead of $\disb_H(U)$.
		Furthermore, we say that a subgraph $F \subset H$ \emph{balances} $U$ if
		$|(U \cap X) \sm V(F)| = |(U \cap Y) \sm V(F)|$.

		To make sure that imbalance is well-defined, we adopt the convention that
		every bipartite graph comes with a prescribed vertex bipartition. This
		choice will usually be clear from the context. For example, $\liftG$ has
		bipartition $\{\partb{V}, \partt{V}\}$ and so does every relevant spanning
		subgraph of $\liftG$.

		Now comes a key definition.  Let $\sigma$ be an ordering of $V(G)$. We define the
		spanning subgraph $\Gsigma$ of $\liftG$ by setting
		\begin{align*}
			E(\Gsigma) =
				\left\{
					\partb{u} \partt{v} \,:\;
					uv \in E(G), \; \sigma(u) < \sigma(v) \text{ and }
					\partb{u},\partt{v} \text{ are in distinct clumps of } \liftG
			\right \}.
		\end{align*}
		The rest of the proof goes as follows. First, we show that there exists an ordering $\sigma$ of $V(G)$ such that $G_\sigma$ contains a so-called balancing matching (see
		\Cref{lem:balancing-matching}). The reason we consider $G_\sigma$ instead of
		working directly with $\liftG$ is that a matching in $G_\sigma$ of size $m$
		corresponds to a linear forest in $G$ of size $m$, whereas the edges of $G$
		corresponding to a matching in $\liftG$ may span a cycle; moreover, an edge
		$uv$ in $G$ may be represented twice in a matching in $\liftG$ -- once as
		$u^{(1)}v^{(2)}$ and once as $u^{(2)}v^{(1)}$. We explain this more
		precisely towards the end of the section. Second, we take the linear forest
		in $G$ that comes from a balancing matching in $G_\sigma$, and we modify it
		slightly so that it satisfies the assertions of
		\Cref{lem:good-linear-forest}.
	\subsection{Balancing $G_\sigma$}

		Here comes the main technical lemma of the section.

		\begin{lem} \label{lem:balancing-matching}
			There is an ordering $\sigma$ such that $\Gsigma$ has a matching $M$ with the following properties:
			\begin{enumerate} \label{lem:balance-Gsigma}
				\item \label{it:M-balances}
					for each $i \in [s]$, $M$ balances $\liftA_i$;
				\item \label{it:M-small}
					$|M| \le (\xi\zeta/8) n$.
			\end{enumerate}
		\end{lem}
		Property~\ref{it:M-balances} is the main part of this lemma: if we find a
		matching in $\Gsigma$ that balances $\liftA_1, \dotsc, \liftA_s$, then we
		get property~\ref{it:M-small} for free from the following argument.

		\begin{prop} \label{prop:small-submatching}
			Let $H$ be a balanced bipartite graph whose vertex set is partitioned into sets
			$U_1, \dotsc, U_k$. Suppose that $M$ is a matching in $H$ that balances
			$U_i$ for every $i \in [k]$. Then $M$ contains a matching that has at
			most $(k-1)(\disb(U_1) + \dotsb + \disb(U_k))$ edges and balances $U_i$
			for every $i \in [k]$.
		\end{prop}

		\begin{proof}
			We use induction on $k$. The base case is when $k = 1$, in which case $U_1 = V(H)$ and $\disb(U_1) = 0$ because $H$ is balanced, so the empty matching is a balancing matching.

			Next, suppose that $k \ge 2$. Denote the bipartition of $H$ by $\{X, Y\}$, and let $X_i = X \cap U_i$, $Y_i = Y \cap U_i$ for $i \in [k]$. Without loss of generality, we assume that $M$ is a minimal matching that balances $U_i$ for every $i \in [k]$.
			We claim that there is $i \in [k]$ for which $M$ does not touch $Y_i$. Indeed, let $D$ be an auxiliary directed graph on vertex set $[k]$, where $ij$ is an edge if there is an \crossedges{X_i}{Y_j} edge in $M$. 
			Suppose that $(i_1 \ldots i_s)$ is a directed cycle in $D$. Then there exist $a_j \in X_{i_j}, b_j \in Y_{i_j}$ such that $a_1 b_2, \ldots, a_s b_1$ are edges of $M$. But then $M \setminus \{a_1 b_2, \ldots, a_s b_1\}$ balances every $U_i$, contradicting the minimality of $M$. It follows that $D$ is acyclic, which implies the existence of $i \in [k]$ with in-degree $0$ in $D$, i.e.\ $M$ does not touch $Y_i$, as claimed. 

			Without loss of generality, suppose that $M$ does not touch $Y_k$. As $M$ balances $U_k$, the number of edges of $M$ that touch $X_k$ is exactly $|X_k| - |Y_k| = \disb(U_k)$. Let $M'$ be the submatching of $M$ obtained by removing the edges that touch $X_k$, let $H'$ be the subgraph of $H$ obtained by removing $U_k$ and vertices of $U_1 \cup \dots \cup U_{k-1}$ that are neighbours of $X_k$ in $M$, and let $U_i' = U_i \cap V(H')$.
			Then $H'$ is a balanced bipartite graph, as exactly $|X_k|$ vertices are removed from each part of $H$ to form $H'$. Moreover, $M'$ is a minimal matching in $H'$ that balances $U_i'$ for every $i \in [k-1]$. Thus, by induction, 
			\[
				|M'| \le (k-2) (\disb(U_1') + \dotsb + \disb(U_{k-1}')) \le (k-2)(\disb(U_1) + \ldots + \disb(U_k)),
			\]
			because the sum of imbalances of $U_1, \ldots, U_{l-1}$ increases by at most $\disb(U_k)$ when going from $H$ to $H'$. Since $|M \setminus M'| = \disb(U_k)$, we have $|M| \le (k-1)(\disb(U_1) + \dotsb + \disb(U_k))$, as required.
		\end{proof}

		\begin{obs} \label{obs:small-disb}
			$\sum_{i=1}^s \disb_{\liftG}(\liftA_i) \le (6 \beta r / c) n$.
		\end{obs}

		\begin{proof}
			Pick $i \in [s]$ and recall that $T_i, B_i$ are the vertex classes of
			$\liftA_i$. Since $\liftG$ is $cn$-regular, we have
			\begin{align*}
				|T_i|cn \;&
				= e(T_i, B_i) + e(T_i, V(\liftG) \sm \liftA_i) \le |B_i|cn \;+ e(T_i, V(\liftG) \sm \liftA_i)
			\end{align*}
			From this upper bound for $|T_i|cn$ and the corresponding upper bound
			for $|B_i|cn$ we get
			\begin{equation*}
				\disb_{\liftG}(\liftA_i) = \big||T_i| - |B_i|\big|\le
				\frac{e(\liftA_i, V(\liftG) \sm \liftA_i)}{cn}
			\end{equation*}
			Summing over all $i$ and applying \Cref{obs:few-crossedges} gives the
			desired result.
		\end{proof}

		\begin{proof}[ of \Cref{lem:balance-Gsigma}]
			As noted above, it is enough to find an ordering $\sigma$ and a
			matching $M \subset \Gsigma$ that satisfies property~%
			\ref{it:M-balances} in \Cref{lem:balancing-matching}. Indeed,
			\Cref{prop:small-submatching,obs:small-disb} then give us a
			submatching of $M$ that satisfies property~\ref{it:M-balances} and
			has at most $(12 \beta r^2 / c) n \le (\xi \zeta / 8) n$ edges, the
			latter bound being a consequence of the assumption that $\beta \ll
			\xi \ll \zeta$. We split our proof into two main steps. In the
			first step we find an ordering $\sigma$ for which there is an
			almost balancing fractional matching in $G_\sigma$. In the second
			step we convert it to a balancing matching in $G_\sigma$. 

			\step{1}{Using the Max-Flow Min-Cut theorem to obtain an almost balancing fractional matching in $\Gsigma$ for some ordering $\sigma$}

				The terms used in the summary of this step are mostly
				self-explanatory, but we define them formally to clarify the
				details. A \emph{fractional matching} in $\Gsigma$ is a function $w$
				that assigns weights from the interval $[0,1]$ to the edges of
				$\Gsigma$ in such a way that for each vertex $v \in V(\Gsigma)$ the
				\emph{weight} of $v$, denoted $w(v)$ and defined as $\sum_{uv \in
				E(\Gsigma)} w(uv)$, does not exceed $1$.
				Let $w$ be a fractional matching in $\Gsigma$. For any $U \subset
				V(\Gsigma)$ we define $w(U) = \sum_{v \in U} w(v)$. For each $i \in
				[s]$ we define $\disb(w,i) = \big| (|T_i| - w(T_i)) - (|B_i| - w(B_i))
				\big|$. We say that $w$ is \emph{\almostbal{\alpha}} if $\sum_{i=1}^s
				\disb(w,i) \le \alpha$. One can think of $w(T_i)$ as the weight of the edges leaving $T_i$ (recall that in $\Gsigma$ there are no \crossedges{T_i}{B_i} edges), and similarly for $B_i$. If $\disb(w, i) = 0$, this means that the fractional matching $w$ balances the cluster $U_i$. Since we are not able to find such a balancing fractional matching directly, we settle for one that is nearly-balancing, and the quantity $\disb(w, i)$ allows us to measure how far $w$ is from balancing $U_i$.
				In this step we will find a \almostbal{0.9} fractional matching in $\Gsigma$ for some $\sigma$.

				We now prepare $\Gsigma$ for an application of the Max-Flow Min-Cut
				theorem, that is, we convert it to a weighted digraph $\Gsigmadi$
				with a source and a sink (see \Cref{fig:flow}). The vertex set of $\Gsigmadi$ contains
				$V(\liftG)$ and $2s+2$ new vertices: source $p$, sink $q$ and, for
				each $i \in [s]$, a pair of new vertices $b_i,t_i$. The edges of
				$\Gsigma$ become arcs of $\Gsigmadi$, directed from $\partb{V}$ to
				$\partt{V}$ (we recall that $\partb{V} = B_1 \cup \dotsb \cup B_s$
				and $\partt{V} = T_1 \cup \dotsb \cup T_s$). For every $i \in [s]$
				we add arcs (1) from $p$ to $b_i$, (2) from $b_i$ to all vertices in
				$B_i$, (3) from all vertices in $T_i$ to $t_i$ and (4) from $t_i$ to
				$q$. Vertices that were present in $\Gsigma$ get capacity $1$, while
				$p,q$ get infinite capacity. The capacities of $b_i,t_i$, $i \in
				[s]$, are defined via quantities $a_{ij}$, $i,j \in [s]$, which we
				now introduce. We set
				\begin{equation*}
					a_{ij} =
					\begin{cases}
						\frac{1}{cn} \, e_{\liftG}(B_i, T_j) &\text{if } i \neq j \\
						0 &\text{otherwise}
					\end{cases}
				\end{equation*}
				and, for every $k \in [s]$,
				\begin{align*}
					b_k \text{ gets capacity } \sum_{j=1}^s a_{kj}, \qquad \qquad 
					t_k \text{ gets capacity } \sum_{i=1}^s a_{ik}.
				\end{align*}

				\begin{figure}[ht]
					\centering
					\includegraphics[scale = .9]{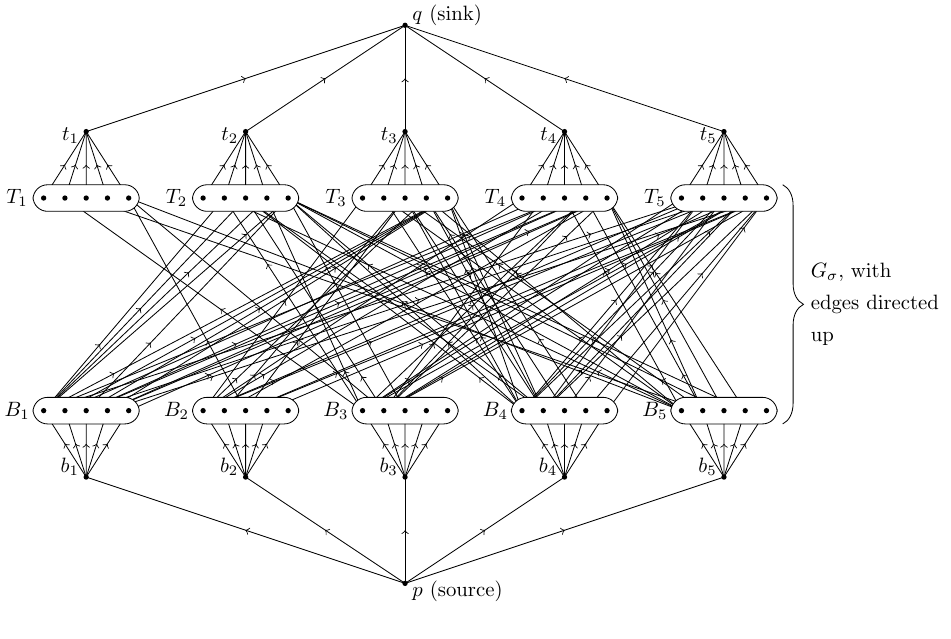}
					\vspace{-18pt}
					\caption{Definition of $\Gsigmadi$.}
					\label{fig:flow}
				\end{figure}

				A \emph{cut} of $\Gsigmadi$ is a subset of $V(\Gsigmadi) \sm
				\{p,q\}$ whose removal from $\Gsigmadi$ disconnects $q$ from $p$.
				
				We will show that, for some $\sigma$, $\Gsigmadi$ does
				not have cuts with capacity less than $\sum_i \sum_j a_{ij} -
				0.9$.  We will then apply the Max-Flow Min-Cut theorem to
				deduce the existence of a flow of at least this value, which in
				turn implies the existence of the required \almostbal{0.9}
				fractional matching. We note that the standard version of
				the Max-Flow Min-Cut theorem places capacities on the arcs than on the vertices and uses an appropriate
				notion of a cut. As the version that we use can
				be proved similarly to the standard one, we elect to use it out of 
				convenience.

				The choice of capacities of the vertices $b_i$ and $t_i$ may seem arbitrary at first glace\footnote{and, indeed, some trial and error was required in order to arrive at this `correct' choice of capacities,}, so before proceeding let us briefly explain why this choice makes sense. For each $i \in [s]$, the difference between the capacity of $b_i$ and the capacity of $t_i$, is the difference between the number of edges of $\liftG$ incident with $B_i$ and the number of edges incident with $T_i$, divided by $cn$, which is exactly the imbalance of the clump $\liftA_i$ in $\liftG$. It follows that a flow in $\Gsigma$ that fully saturates both $b_i$ and $t_i$ (namely, the amount of flow through each of these vertices equals their capacity) translates into a fractional matching in $\liftG$ that balances $\liftA_i$. Thus, a flow in $\Gsigma$ in which $b_i$ and $t_i$ are fully saturated for all $i \in [s]$, translates into the desired balancing fractional matching. Since the value of such a flow is $\sum_i \sum_j a_{ij}$, any flow with almost this value (a proof of whose existence in some $\Gsigma$ will be the main aim of this section) almost saturates the vertices $b_i$ and $t_i$ for each $i \in [s]$, and translates into the required almost-balancing fractional matching.

				With the goal of proving that some $\Gsigma$ has no cuts of low capacity in mind, we consider graphs $F_{I,J,\sigma}$, defined for all $I,J
				\subset [s]$, that are the induced subgraphs of $\Gsigma$ on vertices
				\begin{equation*}
					V(F_{I,J,\sigma}) = \left( \bigcup_{i \in I} B_i \right) \cup 
										\left( \bigcup_{j \in J} T_j \right).
				\end{equation*}
				The point of this definition is that every cut of $\Gsigmadi$
				induces a vertex cover of $F_{I,J,\sigma}$ for appropriately chosen
				$I,J$. This is why the following claim is useful.

				\begin{claim} \label{claim:matching-I-J}
					Fix $I,J \subset [s]$. Let $\sigma$ be a random ordering of
					$V(G)$, chosen uniformly at random.  With probability
					greater than $1-4^{-s}$, every vertex cover of
					$F_{I,J,\sigma}$ contains at least $\sum_{i \in I} \sum_{j
					\in J} a_{ij} - 0.9$ vertices.
				\end{claim}

				\begin{proof}
					We define 
					\begin{equation*}
						E_{I,J} = \left\{ \partb{u}\partt{v} : uv \in E(G) \text{ and }
						\partb{u} \in B_i, \partt{v} \in T_j \text{ with } i \in I, j \in J,
						i \neq j \right\}.
					\end{equation*}
					In other words, $E_{I,J}$ is the set of edges of
					$\liftG[V(F_{I,J,\sigma})]$ that have ends in separate clumps.
					Note that $|E_{I,J}| = cn \sum_{i \in I} \sum_{j \in J} a_{ij}$
					and that any given edge $\partb{u} \partt{v} \in E_{I,J}$ is in
					$F_{I,J,\sigma}$ if and only if $\sigma(u) < \sigma(v)$.
					Furthermore, it follows from \Cref{obs:cluster-min-degree} that
					any vertex in $V(F_{I,J,\sigma})$ is incident with at most $(c -
					\delta/2)n$ edges in $E_{I,J}$.

					We classify the vertices of $F_{I,J,\sigma}$ as rich or poor,
					according to the following rule (which does not depend on
					$\sigma$):
					\begin{equation*}
						v \in V(F_{I,J,\sigma}) \text{ is }
						\begin{cases}
							\text{\emph{rich}} \text{ if $v$ is incident with at
								least $c n / (1000 s)$ edges in $E_{I,J}$} \\
							\text{\emph{poor}} \text{ otherwise}.
						\end{cases}
					\end{equation*}
					We also say that $e \in E_{I,J}$ is \emph{rich} if at least one
					end of $e$ is rich and \emph{poor} otherwise. We write $\Erich$
					and $\Epoor$ to denote the sets of, respectively, rich and poor
					edges in $E_{I,J}$.

					Our strategy is as follows: first, with high probability, we
					construct a matching in $\Epoor \cap E(F_{I,J,\sigma})$ of size
					at least $|\Epoor| / (cn) - 0.9$; then, also with high
					probability, we construct a matching in $\Erich \cap
					E(F_{I,J,\sigma})$ of size at least $|\Erich| / (cn)$, ensuring
					that these two matchings are vertex-disjoint. If we are
					successful in both tasks, then the union of these matchings is a
					matching in $F_{I,J,\sigma}$ of size at least $|E_{I,J}|/(cn) -
					0.9$, giving the desired result.

					First, we deal with the poor edges. Since the smallest vertex
					cover of $\Epoor$ contains only poor vertices, the cardinality
					of such a cover is at least $1000s|\Epoor| / (cn)$. By
					K\H{o}nig's theorem $\Epoor$ contains a matching $M$ of size
					$|M| \ge 1000s|\Epoor| / (cn)$. We say that two distinct
					edges $e,f \in M$ are \emph{related} if there exists a vertex $u
					\in V(G)$ such that $\partb{u}$ is an end of $e$ and $\partt{u}$
					is an end of $f$, or vice versa. We greedily construct a subset
					$M' \subset M$ such that $|M'| \ge |M|/3$ and $M'$ does not
					contain any pairs of related edges: initially we set $M' = \es$
					and consider the edges in $M$ one by one, putting $e \in M$ into
					$M'$ if $e$ is not related to any edges already present in $M'$.
					The bound $|M'| \ge |M| / 3$ comes from the fact that any edge
					of $M$ is related to at most two other edges. Indeed, for every
					edge $e \in M \sm M'$ there exists an edge in $M'$ that
					prevented $e$ from being accepted into $M'$, while a single edge
					in $M'$ can prevent at most two edges from being accepted,
					giving $|M \sm M'| \le 2 |M'|$.

					Let $\Evpoor$ be the event that $|M' \cap E(F_{I,J,\sigma})| \ge
					|\Epoor| / (cn) - 0.9$.  A given edge $\partb{u} \partt{v} \in
					M'$ is in $E(F_{I,J,\sigma})$ if and only if $\sigma(u) <
					\sigma(v)$, which happens with probability $1/2$.  Moreover,
					since $M'$ does not contain related edges, the events of
					particular edges of $M'$ being present in $E(F_{I,J,\sigma})$
					are independent, because they are determined by restrictions of
					$\sigma$ to mutually disjoint pairs of vertices. As a result,
					$|M' \cap E(F_{I,J,\sigma})|$ has distribution
					$\text{Binom}(|M'|, 1/2)$. An application of a Chernoff's bound
					gives
					\begin{equation*} 
						\Pb{\big|M' \cap E(F_{I,J,\sigma})\big|
						< \frac{|M'|}{3}}
						\le \exp\left( -\frac{|M'|}{18} \right).  
					\end{equation*}

					Note that, in particular, $|M'|/3 \ge 1000s|\Epoor| / (9 c n)
					\ge |\Epoor| / (cn)$.  If $|\Epoor| \ge (162 / 1000) c n$, then
					we also have $|M'| \ge 54 s$, and hence $\Evpoor$ holds with
					probability at least $1 - \exp(-3s) > 1 - 4^{-s}/2$. On the
					other hand, if $|\Epoor| < (162 / 1000) c n$, then $|\Epoor| /
					(cn) < 0.9$, which means that $\Evpoor$ trivially holds. In
					either case,
					\begin{equation*}
						\mathbb{P}(\Evpoor) > 1 - \frac{4^{-s}}{2}.
					\end{equation*}

					We now  turn our focus to the rich edges. First, suppose that
					$\Erich \neq \es$. Since any vertex in $V(F_{I,J,\sigma})$ is
					incident with at most $(c - \delta/2) n$ edges in $E_{I,J}$,
					there are at least $|\Erich| / (cn - \delta n / 2)$ rich
					vertices. Let $\ell = \lceil |\Erich| / (cn - \delta n/2)
					\rceil$ and let $R$ be a set of $\ell$ rich vertices. We say
					that a vertex in $R$ is \emph{ruined} if its degree in
					$F_{I,J,\sigma}$ is smaller than $\delta\sqrt\beta n$. Consider an
					arbitrary vertex in $R$ that belongs to the vertex class
					$\partb{V}$, that is, a vertex of the form $\partb{v} \in R$
					with $v \in V(G)$.  Let $u_1, \dotsc, u_d$ be the vertices in
					$V(G)$ such that $\partt{u_1}, \dotsc, \partt{u_d}$ are adjacent
					to $\partb{v}$ via edges in $E_{I,J}$.  Since $\partb{v}$ is
					rich, $d \ge cn/(1000s)$. Note that $\partb{v}$ is ruined if and
					only if $v$ appears in one of the final $\lceil \delta\sqrt\beta
					n \rceil$ positions of the order that $\sigma$ induces on $\{ v,
					u_1, \dotsc, u_d \}$. Since $v$ is equally likely to be in any
					position of this order, we have
					\begin{equation*}
						\Pb{\partb{v} \text{ is ruined}} 
						\le \frac{\delta\sqrt\beta n + 1}{c n/(1000s) + 1} 
						\le \frac{2000s\delta\sqrt\beta}{c}
						< \frac{\delta}{4^{s+1} c},
					\end{equation*}
					where the latter inequality comes from the assumption that
					$\beta \ll \cmin$.  The same bound holds for those vertices in $R$
					that are in the vertex class $\partt{V}$.  Hence, the expected
					number of ruined vertices in $R$ is at most $4^{-s-1} \delta
					\ell / c$.  Markov's inequality gives
					\begin{equation*}
						\Pb{R \text{ has at least } \frac{\delta}{2c} \ell
						\text{ ruined vertices}}
						< \frac{4^{-s}}{2}.
					\end{equation*}

					Let $\Evrich$ be the event that at least $|\Erich| / (cn)$
					vertices in $R$ are not ruined. If $\Erich = \es$, then
					$\Evrich$ trivially holds. Otherwise, as we have just seen, with
					probability greater than $1 - 4^{-s}/2$, there are at least $(1 -
					\delta/(2c))\ell$ vertices in $R$ that are not ruined. Since $\ell \ge
					|\Erich| / ( cn - \delta n / 2 )$, we have $\mathbb{P}(\Evrich)
					> 1 - 4^{-s}/2$.
					
					At this point we have established that $\mathbb{P}(\Evpoor \cap
					\Evrich) > 1 - 4^{-s}$.  We will finish the proof of the claim
					by assuming that $\Evpoor, \Evrich$ both occur and constructing
					a matching in $F_{I,J,\sigma}$ of size at least $|\Epoor|/(cn) +
					|\Erich|/(cn) - 0.9$.  From $\Evpoor$ we get a matching $M_0
					\subset \Epoor \cap E(F_{I,J,\sigma})$ of size $|M_0| \ge
					|\Epoor| / (cn) - 0.9$.  Furthermore, since $\Evrich$ occurs,
					there exist $m = \lceil |\Erich| / (cn) \rceil$ distinct rich
					vertices $v_1, \dotsc, v_m \in V(F_{I,J,\sigma})$ of degree at
					least $\delta\sqrt\beta n$ in $F_{I,J,\sigma}$. Note that $v_1,
					\dotsc, v_m \not\in V(M_0)$ because the edges in $M_0$ are poor.

					We now construct an eventually terminating sequence of matchings
					$M_0 \subset M_1 \subset \dotsb$ in $F_{I,J,\sigma}$, where
					$M_{i+1}$ is obtained by adding to $M_i$ a single edge incident
					with $v_{i+1}$.  Suppose that we have just constructed $M_i$ for
					some $i \ge 0$. If $|M_i| \ge |\Epoor|/(cn) + |\Erich|/(cn) -
					0.9$, then we stop. If not, then we have $i \le m-1$, because
					$|M_i| = |M_0| + i$. Since $v_{i+1}$ has at least
					$\delta\sqrt\beta n$ neighbours in $F_{I,J,\sigma}$, we can pick
					one, say $u_{i+1}$, that is not contained in $V(M_i) \cup
					\{v_{i+2}, \dotsc, v_m\}$ (here we use the bound $|V(M_i)| + m
					\le 3|E_{I,J}| / (cn) + 1 \le (9 r \beta / c) n + 1 <
					\delta\sqrt\beta n$, which is a consequence of
					\Cref{obs:few-crossedges} and the assumption that $\beta \ll
					\delta$).  The new matching $M_{i+1}$ is defined as $M_i \cup \{
					v_{i+1} u_{i+1} \}$. We remark that our construction ensures
					that at each stage $v_{i+1}$ is not contained in $V(M_i)$, and
					so the process keeps running until we obtain a matching of a
					desired size.
					\Cref{claim:matching-I-J} is proved.
				\end{proof}
				
				\begin{claim}
					There exists $\sigma$ for which the capacity of every cut of
					$\Gsigmadi$ is at least $\sum_{i = 1}^s \sum_{j = 1}^s a_{ij} -
					0.9$.
				\end{claim}

				\begin{proof}
					Let $\sigma$ be a random ordering, chosen uniformly at random.
					For any $I,J \subset [s]$, let $\mathcal{E}_{I,J}$ be the event
					that $F_{I,J,\sigma}$ has no vertex cover of cardinality less
					than $\sum_{i \in I} \sum_{j \in J} a_{ij} - 0.9$. We know from
					the previous claim that $\mathbb{P}(\mathcal{E}_{I,J}) > 1 -
					4^{-s}$ for any $I,J$. Since there are $4^s$ choices for $I,J$,
					all events $\mathcal{E}_{I,J}$ occur simultaneously with
					positive probability.

					Suppose that $\mathcal{E}_{I,J}$ occurs for every $I, J \subset
					[s]$ and let $C$ be a cut of $\Gsigmadi$. Then
					$\mathcal{E}_{I,J}$ holds in particular for the choice $I = \{i
					\in [s] : b_i \not\in C\}, J = \{j \in [s] : t_j \not\in C\}$.
					Since $C$ disconnects $q$ from $p$, it in particular intersects
					all paths from $p$ to $q$ that visit $(\bigcup_{i \in I} B_i)
					\cup (\bigcup_{j \in J} T_i) = V(F_{I,J,\sigma})$, and hence $C
					\cap V(F_{I,J,\sigma})$ is a vertex cover of $F_{I,J,\sigma}$.
					Therefore,
					\begin{align*}
						\capacity(C)
						&= \sum_{i \not\in I} \capacity(b_i) +
						   \sum_{j \not\in J} \capacity(t_j) +
						   |C \cap V(F_{I,J,\sigma})| \\
						&\ge \sum_{i \not\in I} \sum_j a_{ij} +
						   \sum_i \sum_{j \not\in J} a_{ij} +
						   \sum_{i \in I} \sum_{j \in J} a_{ij} - 0.9 \\
						&\ge \sum_i \sum_j a_{ij} - 0.9,
					\end{align*}
					where the first inequality follows from the assumption that
					$\mathcal{E}_{I,J}$ occurs.
				\end{proof}

				We fix one instance of $\sigma$ for which the capacity of a minimum
				cut of $\Gsigmadi$ is at least $\sum_i \sum_j a_{ij} - 0.9$. The
				Max-Flow Min-Cut Theorem produces a flow $f$ on $\Gsigmadi$ with
				$\val(f) \ge \sum_i \sum_j a_{ij} - 0.9$. This flow induces a
				fractional matching in $\Gsigma$, as the capacity of vertices in $\Gsigma$ was set to $1$. Abusing the notation slightly, we
				denote this fractional matching also by $f$.

				\begin{claim}
					The fractional matching $f$ is \almostbal{0.9}.
				\end{claim}

				\begin{proof}
					It is clear from the way the directed graph $\Gsigmadi$ was set
					up that, for every $i \in [r]$, $f(B_i)$ does not exceed the
					capacity of $b_i$ in $\Gsigmadi$. That is, $\sum_j a_{ij} -
					f(B_i) \ge 0$.  Therefore,
					\begin{equation*}
						\sum_i \sum_j a_{ij} - 0.9
						\,\le\, \val(f)
						\,=\, f(B_1) + \dotsb + f(B_r)
						\,\le\, \sum_i \sum_j a_{ij},
					\end{equation*}
					from which we deduce that
					\begin{equation*}
						0 \le \sum_i \left( \sum_j a_{ij} - f(B_i) \right) \le
						0.9.
					\end{equation*}
					Similarly, we have $\sum_i a_{ij} - f(T_j) \ge 0$ for all
					$j$ and
					\begin{equation*}
						0 \le \sum_j \left( \sum_i a_{ij} - f(T_j) \right) \le
						0.9.
					\end{equation*}

					At this point it is important to remember that for all distinct
					$i,j$ we have $a_{ij} \, cn = e_{\liftG}(B_i, T_j)$. Also,
					$a_{ii} = 0$. Since $\liftG$ is $cn$-regular, for each $k \in
					[s]$ we have $|B_k| cn - |T_k| cn = \sum_j a_{kj} \, cn - \sum_i
						a_{ik} \, cn$,
					which can be rearranged to give
					\begin{align*}
						\disb(f,k)
						&= \left| \left( \sum_j a_{kj} - f(B_k) \right) -
								  \left( \sum_{i\vphantom{j}} a_{ik} - f(T_k)
								  \right)
						   \right| \\
						&\le \max \left\{ \sum_j a_{kj} - f(B_k), \;
										  \sum_i a_{ik} - f(T_k)
								  \right\}.
					\end{align*}
					Therefore, 
					\begin{equation*}
						\sum_k \disb(f,k) \le \max \left\{ \sum_k
						\left(\sum_j a_{kj} - f(B_k)\right), \sum_k
						\left(\sum_{i\vphantom{j}} a_{ik} - f(T_k)\right) \right\}
						\le 0.9,
					\end{equation*}
					as claimed.
				\end{proof}

				\step{2}{Converting the almost balancing fractional matching to a
				balancing matching}

				Let $w$ be any fractional matching in $\Gsigma$. We say that a
				vertex $v \in V(\Gsigma)$ is \emph{open} if $w(v) \in (0,1)$ and
				\emph{closed} if $w(v) \in \{0,1\}$. Similarly, we say that an edge
				$e \in E(\Gsigma)$ is \emph{open} if $w(e) \in (0,1)$ and
				\emph{closed} if $w(e) \in \{0,1\}$.

				We know that $\Gsigma$ has a \almostbal{0.9} fractional matching,
				namely, $f$. Let $f^\ast$ be a \almostbal{0.9} fractional matching
				in $\Gsigma$ that minimises the total number of open vertices and
				open edges.

				\begin{claim}
					The fractional matching $f^\ast$ is \almostbal{0} and has
					integer weights.
				\end{claim}

				\begin{proof}
					It suffices to show that $f^\ast$ has no open edges. Indeed,
					this would imply that $f^\ast$ has no open vertices, and so its
					imbalance is a whole number. However, by definition,
					$\disb(f^\ast) \le 0.9$, and so $\disb(f^\ast) = 0$, as
					required.

					We assume for contradiction that $f^\ast$ has at least one open
					edge.  Let $E_{\text{open}}$ and $V_{\text{open}}$ stand for the
					sets of, respectively, open edges and open vertices of $f^\ast$.
					There may be closed vertices that are incident with open edges;
					we call such vertices \emph{full} and denote their set by
					$V_{\text{full}}$.  Clearly, full vertices have weight $1$ and
					are incident with at least two open edges. We will now add new
					edges, which we call \emph{fake}, to $\Gsigma$. For every $i \in
					[s]$, we add a path spanning the open vertices contained in
					$\liftA_i$. In particular, if for some $i$ there is at most one
					open vertex in $\liftA_i$, then we do not create any fake edges
					in the clump $\liftA_i$. Let $E_{\text{fake}}$ stand for the set
					of fake edges that were added to $\Gsigma$.

					We create an auxiliary graph $H$ with vertices $V_{\text{open}}
					\cup V_{\text{full}}$ and edges $E_{\text{open}} \cup
					E_{\text{fake}}$.  First, suppose that $H$ contains a cycle $C$.
					Since $E_{\text{fake}}$ is a union of vertex-disjoint paths, $C$
					must contain at least one open edge. Fix a direction for $C$.
					For $e$ an open edge in $C$, we say that $e$ is \emph{upward} if
					it is directed from $\partb{V}$ to $\partt{V}$.  If $e$ is
					directed from $\partt{V}$ to $\partb{V}$, then we say that $e$
					is \emph{downward}.

					Let $\lambda > 0$ be a small positive number and let
					$f^\ast_\lambda$ be the fractional matching in $\Gsigma$
					obtained from $f^\ast$ by adding $\lambda$ to the weight of
					every upward open edge and subtracting $\lambda$ from the weight
					of every downward open edge.  We remark that $f^\ast_\lambda$ is
					a valid fractional matching, provided that $\lambda$ is small
					enough so that the modified weights of open edges and open
					vertices remain in the interval $[0,1]$; crucially, each full
					vertex in $C$ is incident with precisely one upward and one
					downward open edge, so its weight remains $1$. Moreover, we
					claim that $f^\ast_\lambda$ is \almostbal{0.9}. In fact, for
					every $i \in [s]$ we have $\disb(f^\ast_\lambda,i) =
					\disb(f^\ast,i)$. This can be seen by observing that every open
					edge in $C$ that enters the clump $B_i \cup T_i$ either
					contributes an additional $\lambda$ term to
					$f^\ast_\lambda(T_i)$ (if it is upward) or an additional
					$-\lambda$ term to $f^\ast_\lambda(B_i)$ (if it is downward) and
					so its added contribution to $f^\ast_\lambda(T_i) -
					f^\ast_\lambda(B_i)$ is $\lambda$. However, the next open edge
					along $C$ leaves $B_i \cup T_i$ and, by similar reasoning, its
					added contribution to $f^\ast_\lambda(T_i) -
					f^\ast_\lambda(B_i)$ is $-\lambda$.  The contributions cancel
					out. We conclude that $\disb(f^\ast_\lambda, i) = \disb(f^\ast,
					i)$, as claimed.  As $\lambda$ increases, eventually a point is
					reached where some open vertex or some open edge becomes closed.
					At that exact moment $f^\ast_\lambda$ has fewer open vertices
					and/or open edges than $f^\ast$, contradicting the minimality of
					$f^\ast$. Therefore, $H$ does not have cycles.

					Since $H$ is a non-empty forest, there exists a path $P$ joining
					two distinct vertices of degree $1$ in $H$, say $x$ and $y$. Suppose
					that $x \in \liftA_i$, $y \in \liftA_j$.  Since $x$ and $y$ have
					degree $1$ in $H$, they are not full and they are not incident
					with fake edges, which means that $x$ and $y$ are the unique open
					vertices in their respective clumps $\liftA_i$ and $\liftA_j$. In
					particular, $i \neq j$ and $P$ contains an open edge. Also,
					precisely one of $f^\ast(B_i)$ and $f^\ast(T_i)$ is an integer,
					and so $\disb(f^\ast, i) > 0$.  Similarly, $\disb(f^\ast, j) >
					0$.  Like in the case where $H$ had a cycle, we fix a direction
					for $P$ and partition the open edges in $P$ into \emph{upward}
					and \emph{downward} ones, depending on whether they go from
					$\partb{V}$ to $\partt{V}$ or the other way around. Let $\lambda
					\in \R$ be a number with small absolute value and define
					$f^\ast_\lambda$ in the same way as previously, that is, by
					giving the upward edges of $P$ additional weight $\lambda$ and
					downward edges $-\lambda$. With the same reasoning as before,
					$f^\ast_\lambda$ is a valid fractional matching provided that
					$\lambda$ is small.  Moreover, for every $m \in [s] \sm \{i,j\}$
					we have $\disb(f^\ast_\lambda,m) = \disb(f^\ast,m)$, also by an
					identical argument. However, the added contributions to
					$\disb(f^\ast_\lambda, i)$ and $\disb(f^\ast_\lambda, j)$ are
					non-zero.  In fact, having the additional $\pm\lambda$ term
					either decreases or further increases the imbalance of the
					clumps $\liftA_i, \liftA_j$ by exactly $|\lambda|$. More
					precisely, there exist constants $s_i, s_j \in \{-1, 1\}$ such
					that, for small $|\lambda|$, $\disb(f^\ast_\lambda,i) =
					\disb(f^\ast,i)+s_i\lambda$ and $\disb(f^\ast_\lambda,j) =
					\disb(f^\ast,j)+s_j\lambda$. Therefore, for small $|\lambda|$,
					$\disb(f^\ast_\lambda) = \disb(f^\ast) + (s_i + s_j) \lambda$.
					Depending on the sign of $s_i + s_j$ we choose $\lambda$ to be
					positive or negative, ensuring that $\disb(f^\ast_\lambda) \le
					\disb(f^\ast) \le 0.9$, which means that $f^\ast_\lambda$ is
					\almostbal{0.9}. Finally, we keep increasing the magnitude of
					$\lambda$ until some open vertex or open edge becomes closed.
					(Here it is important to note that the signs $s_i, s_j$ cannot
					change before at least one open vertex become closed, at which
					time we stop our process.) This contradicts the
					minimality of $f^\ast$.  Therefore, the auxiliary graph $H$ is
					empty, and the claim follows.
				\end{proof}

				Since all weights of $f^\ast$ are $0$ or $1$, $f^\ast$ gives rise to
				a matching $M$ in $\Gsigma$. Furthermore, since $f^\ast$ is
				\almostbal{0}, $M$ balances $\liftA_1, \dotsc, \liftA_s$.
				\Cref{lem:balance-Gsigma} follows.
		\end{proof}

	\subsection{Constructing the balancing paths in $G$}

	 	In this section we prove \Cref{lem:good-linear-forest}. Before turning to the proof, we mention the following proposition. A similar result can be deduced, for example, from Lemma 5.4 in \cite{debiasio-nelsen}. We include a short proof, for completeness, in \Cref{appen:connecting}.

		\begin{prop} \label{prop:short-connection}
			Let $\zeta \in (0,1)$ and let $H$ be a graph that has no $\zeta$-sparse cuts.
			Then, for any $R \subset V(H)$ with $|R| \le (\zeta / 6) |H|$ and any
			distinct vertices $x,y \in V(H) \sm R$, there exists a path in $H \sm R$ of
			length at most $3 / \zeta$, with ends $x$ and $y$.
		\end{prop}

		\begin{proof}[ of \Cref{lem:good-linear-forest}]

			The rough idea is as follows. We pull back a balancing matching $M$ of
			$\liftG$, as given by \Cref{lem:balance-Gsigma}, to $G$. The resulting
			subgraph $H_0 \subset G$ has maximum degree at most $2$, is acyclic
			and `overbalances' every \nearb{\beta} cluster $A_i$ (the reason for
			this is that every \nearb{\beta} $A_i$ gives rise to two clumps of
			$\liftG$, both of which are balanced by $M$; therefore, $A_i$ gets
			balanced `twice'). Since $M$ is small, $H_0$ is also small, but it may
			have many components and, as a result, many leaves. To obtain
			property~\ref{it:H-two-ends} in \Cref{lem:good-linear-forest}, in clusters with too many such
			vertices, we connect pairs of them by short paths. It turns out that in
			doing so we also fix the overbalancing issue.  Therefore, we get
			properties~\ref{it:H-two-ends}~and~\ref{it:H-balances}
			simultaneously.  The remaining three properties are mainly
			technicalities.  We use \Cref{prop:short-connection} to find the desired short paths in clusters.

			Fix an ordering $\sigma$ of $V(G)$ such that $\Gsigma$ contains a
			matching $M$ as given by \Cref{lem:balance-Gsigma}; that is, $M$ covers
			at most $|M| \le (\xi \zeta / 4) n$ vertices and, for each $i \in
			[r]$, it satisfies $|T_i \sm V(M)| = |B_i \sm V(M)|$.  Let $H_0$ be the
			subgraph of $G$ spanned by edges $uv \in E(G)$ for which $\partb{u}
			\partt{v}$ or $\partb{v} \partt{u}$ is in $M$. By construction of
			$G_\sigma$, it is impossible for both $\partb{u}\partt{v}$ and
			$\partb{v}\partt{u}$ to be in $M$, and therefore $e(H_0) = e(M)$.
			Trivially, $H_0$ has no isolated vertices.  Moreover, $H_0$ does not
			have cycles. Indeed, suppose to the contrary that $H_0$ contains a cycle
			$v_1 \dots v_\ell$. We may assume that $\partb{v_1}\partt{v_2}$ is in
			$M$. Since $M$ is a matching, $\partt{v_2}\partb{v_3} \not\in M$, and
			hence $\partb{v_2}\partt{v_3} \in M$. Similarly,
			$\partb{v_3}\partt{v_4}, \dotsc, \partb{v_{\ell-1}}\partt{v_\ell},
			\partb{v_\ell}\partt{v_1}$ are edges in $M$. However, this implies that
			$\sigma(v_1) < \dotsb < \sigma(v_\ell) < \sigma(v_1)$, giving a
			contradiction. 
			
			We now show that the number of leaves of $H_0$ in $A_i$ is even for
			every $i \in [r]$. For any subgraph $F \subset G$ and any set $U \subset
			V(G)$ we define $d_F(U) = \sum_{v \in U} d_F(v)$.  We claim that
			$d_{H_0}(A_i)$ is even for every $i \in [r]$. Indeed, if $A_i$ is
			\farb{\gamma}, then $d_M(A_i^{(1)}) = d_M(A_i^{(2)})$, as $M$ balances
			the balanced bipartite graph with bipartition $\{A_i^{(1)},
			A_i^{(2)}\}$, thus implying that $d_{H_0}(A_i) = d_M(A_i^{(1)}) +
			d_M(A_i^{(2)}) = 2d_M(A_i^{(1)})$. Now suppose that $A_i$ is
			\nearb{\beta} and denote its prescribed bipartition by $\{X_i, Y_i\}$.
			Since $M$ balances the two bipartite graphs with bipartitions
			$\{X_i^{(1)}, Y_i^{(2)}\}$ and $\{Y_i^{(1)}, X_i^{(2)}\}$, we have
			$d_{H_0}(X_i) - d_{H_0}(Y_i) = 2(|X_i| - |Y_i|)$. Either way, we see
			that $d_{H_0}(A_i)$ is even.  Since all non-leaves in $H_0$ have 
			degree $2$, we find that the number of leaves of $H_0$ in $A_i$
			is even, as desired.

			We proceed by extending $H_0$ to linear forests $H_0 \subset H_1 \subset
			\dotsb \subset H_m$ (for some $m \ge 0$) where, for each $j \in [m]$,
			$H_j$ is obtained from $H_{j-1}$ by adding a short path contained in
			some cluster $A_i$, joining two leaves of $H_{j-1}$. We stop when we
			reach a linear forest $H_m$ that satisfies
			property~\ref{it:H-two-ends}. 

			Here is a more precise description of this process. Suppose that we have
			constructed linear forests $H_0 \subset H_1 \subset \dotsb \subset
			H_{j-1}$ where $H_t$ contains an even number of leaves in $A_i$ for
			every $t \in \{0, \dotsc, j-1\}$ and every $i \in [r]$. Suppose that
			$H_{j-1}$ does not satisfy property~\ref{it:H-two-ends}. For
			convenience, we write $L = \{ v \in V(H_{j-1}) : v \text{ is a leaf of }
			H_{j-1} \}$. We pick $i \in [r]$ such that $|A_i \cap L| \neq 0, 2$, so
			$|A_i \cap L| \ge 4$.  Since every component of $H_{j-1}$ is a path (and
			so contains two leaves), there exist vertices $x,y \in A_i \cap L$ that
			are in different components of $H_{j-1}$.  By
			\Cref{prop:short-connection}, $G[A_i]$ contains a path $P_j$ of length at
			most $3 / \zeta$, with ends $x,y$ and whose vertex set does not
			intersect $V(H_{j-1}) \sm \{x,y\}$.  We set $H_j = H_{j-1} \cup P_j$ and
			note that our way of choosing $x,y$ ensures that $H_j$ is a linear
			forest. Moreover, since the set of leaves of $H_j$ is the set of leaves
			of $H_{j-1}$ minus $\{x,y\}$, the property that every cluster contains
			an even number of leaves still holds. This also implies that eventually
			we will find a linear forest $H_m$ that satisfies
			property~\ref{it:H-two-ends}.

			To justify the application of \Cref{prop:short-connection} in the
			previous paragraph, we note that, by our inductive construction,
			$|H_{j-1}| \le |H_0| + (3 / \zeta)(j-1)$. Moreover, since $H_{j-1}$ has
			$2(j-1)$ fewer leaves than $H_0$, we have $|H_0| - 2(j-1) \ge 0$, which
			implies that $j-1 \le |H_0|/2$, and therefore $|H_{j-1}| \le
			|H_0|(3/(2\zeta) + 1) \le (\xi \zeta / 4)(2 / \zeta) n^2 \le (\zeta /
			6) n$, as needed.

			It is clear that $H_m$ satisfies
			properties~\ref{it:H-no-isolated}~and~\ref{it:H-two-ends}. Also, by
			the same argument as above, $|H_m| \le (\xi\zeta/4)(2 / \zeta) n \le
			(\xi/2) n$. We now focus on modifying $H_m$ so that it also satisfies
			properties~\ref{it:H-ends-equally}~and~\ref{it:H-balances}. Let $i
			\in [r]$ be the index of an arbitrary \nearb{\beta} cluster $A_i$ and
			denote the prescribed bipartition of $A_i$ by $\{X_i, Y_i\}$.  First,
			suppose that $X_i$ and $Y_i$ have the same number $t$ of
			leaves of $H_m$, so $t \in \{0,1\}$. Then $|V(H_m) \cap X_i| - |V(H_m) \cap Y_i| =
			(d_{H_m}(X_i) + t)/2 - (d_{H_m}(Y_i) + t)/2 = |X_i| - |Y_i|$, as
			$d_{H_0}(X_i) - d_{H_0}(Y_i) = 2(|X_i| - |Y_i|)$ (since $M$ balances the
			two clumps corresponding to $A_i$) and $d_{H_j}(X_i) - d_{H_j}(Y_i)$ is
			the same for all $j \in [m]$, because for each $j \in [m]$, the path $P_j$ that is added to $H_{j-1}$ to form $H_j$ is contained in one of the \nearb{\beta} clusters $A_i$, and thus $d_{P_j}(X_i) = d_{P_j}(Y_i)$ for each $i \in [r]$ such that $A_i$ is \nearb{\beta}. It follows that
			properties~\ref{it:H-ends-equally}~and~\ref{it:H-balances} hold in
			this case. So, without loss of generality, we assume that both leaves of $H_m$ are in $X_i$
			and denote them by $x, x'$. Since $x$ has at least $(\delta / 2) n >
			|H_m|$ neighbours in $Y_i$, it has a neighbour $y \in Y_i \sm V(H_m)$.
			We define $e_i = xy$, with the intention of adding this edge to $H_m$ to
			obtain the desired linear forest $H$. Clearly, $d_{H_m \cup
			\{e_i\}}(X_i) - d_{H_m \cup \{e_i\}}(Y_i) = d_{H_m} (X_i) - d_{H_m}
			(Y_i)$, $x'$ is the unique leaf of $H_m \cup \{e_i\}$ in $X_i$ and $y$
			is the unique such vertex in $Y_i$. The same calculation as in the
			previous case gives $|V(H_m \cup \{e_i\}) \cap X_i| - |V(H_m \cup
			\{e_i\}) \cap Y_i| = |X_i| - |Y_i|$.

			The final definition of $H$ is as follows: it is the subgraph of $G$
			spanned by the edges $E(H_m) \cup \{ e_i : i \in [r] \text{ is such
			that } e_i \text{ is defined}\}$. It follows from the construction of
			$H_m$ and the $e_i$'s that $H$ is a linear forest satisfying
			properties~\ref{it:H-no-isolated}-\ref{it:H-balances}.  Furthermore,
			$|H| \le |H_m| + r \le (\xi/2) n + r \le \xi n$, and so
			property~\ref{it:H-small} also holds.  This completes the proof of
			\Cref{lem:good-linear-forest}.
		\end{proof}

\section{Concluding remarks} \label{sec:conclusion}

	In this paper we prove that the vertices of every $d$-regular $n$-vertex
	graph, where $d \ge cn$ and $n \ge n_0(c)$, can be partitioned into at most
	$\floor{n / (d+1)}$ cycles. It is natural to wonder whether this lower bound
	on $d$ can be lowered. We believe that, with our methods, one could prove
	this result for $d \ge cn/ \sqrt{\log \log n}$. Indeed, the improvement comes from an improved version of \Cref{thm:robust-expanders-hamilton}, proved by Lo and Patel \cite{lo-patel}, which allows for the robust out-expander to have minimum semi-degree $n^{1 - 1/13}$, as well as an improved version of \Cref{lem:good-partition} (which we do not present here, but the proof should be similar). Interestingly, the main obstruction to further lowering the lower bound on the degree appears to be  \Cref{lem:good-partition}.\footnote{The reason is that in our proof we consider a sequence $\mu_1, \ldots, \mu_r$, where we need $\mu_1$ to be at least $\mu_r^{C^{r^2}}$, for some (large) constant $C$. Since we also require that $\mu_r$ is bounded away from $1$ and $\mu_1 \ge 1/n^2$, we obtain $r = O(\sqrt{\log \log n})$, so for our proof to work we need $d = O(n / \sqrt{\log \log n})$.}
	In particular, we think that if \Cref{lem:good-partition} could be strengthened to allow for a degree as small as $n^{1 - \eps}$, for a constant $\eps > 0$, then the main result for regular graphs with degree at least $n^{1 - \delta}$, for a constant $\delta > 0$, would follow. A solution of this problem for much smaller $d$, say $d = \sqrt{n}$, seems to be out of reach.

	It would also be interesting to determine if a version of our results holds
	for regular directed graphs or for regular oriented graphs.
	Another possible direction is to consider bipartite
	versions of the Bollob\'as and H\"aggkvist conjecture (see \cite{bollobas,jackson}). H\"aggkvist
	\cite{haggkvist-bip} conjectured that every bipartite $d$-regular $2$-connected
	bipartite graph on $n$ vertices, where $d \ge n/6$, is Hamiltonian. This was
	essentially verified by Jackson and Li \cite{jackson-li-bip} who proved this
	statement for $d \ge (n+38)/6$. Recently, Li \cite{li-bip} conjectured that
	every $d$-regular $3$-connected bipartite graph on $n$ vertices, with $d \ge
	n/8$, is Hamiltonian. We suspect that our methods could be useful for this
	problem.

\subsection*{Acknowledgements}
	We would like to thank the anonymous referees for their helpful comments; in particular, suggestions made by one of the referees allowed us to significantly shorten the proof of \Cref{lem:rob-ham}.

	\bibliography{reg-path}
	\bibliographystyle{amsplain}
\appendix

\section{Proof of \Cref{prop:short-connection}} \label{appen:connecting}

	\begin{proof}[ of \Cref{prop:short-connection}]
		Fix $R \subset V(H)$ with $|R| \le (\zeta / 6) |H|$ and let $x,y \in V(H)
		\sm R$ be distinct vertices.  We first observe that $H \sm R$ is connected.
		Indeed, if $V(H) \sm R$ admits a partition into non-empty sets $X, Y$ with
		no \crossedges{X}{Y} edges, then the number of \crossedges{X}{(Y \cup R)}
		edges is at most $|X||R|$. We may assume that $|Y| \ge |X|$, and hence that
		$|Y \cup R| \ge |H|/2$, which implies that $|R| \le (\zeta / 3) |Y \cup R|$.
		However, this contradicts the assumption that the number of
		\crossedges{X}{(Y \cup R)} edges in $H$ is at least $\zeta |X||Y \cup R|$.

		Now, we partition the vertices of $H \sm R$ into sets according to their
		distance to $x$. That is, for all $i \ge 0$ we set 
		\begin{equation*}
			L_i = \{v \in V(H) \sm R : \text{the shortest path from } x \text{ to }
			v \text{ in } H \sm R \text{ has } i \text{ edges} \}
		\end{equation*}
		Since $H \sm R$ is finite and connected, there exists a maximum value $a$
		for which $L_a$ is non-empty and, for that value, $L_0, \dotsc, L_a$
		partition $V(H) \sm R$.

		Our aim is to show that $a \le 3 / \zeta$, so suppose that this is not the
		case. In particular, we have $a \ge 3$. Let $j$ be an index in the set
		$[a-1]$ for which $|L_j|$ is minimal. We partition $V(H) \sm R$ into two
		sets $X,Y$, defined by
		\begin{equation*}
			\begin{cases}
				\text{if } j \ge \frac{a}{2}, \text{ then }
					X = L_0 \cup \dotsb \cup L_j \text { and }
					Y = L_{j+1} \cup \dotsb \cup L_a, \\
				\text{if } j < \frac{a}{2}, \text{ then }
					X = L_j \cup \dotsb \cup L_a \text { and }
					Y = L_0 \cup \dotsb \cup L_{j-1}.
			\end{cases}
		\end{equation*}
		In either case $X,Y$ are non-empty sets such that there are no edges between
		$X \sm L_j$ and $Y$. Moreover, $X$ contains at least $a/2$ of the sets $L_1,
		\dotsc, L_{a-1}$, and so $|X| \ge |L_j|a/2$. Therefore, the number of
		\crossedges{X}{Y} edges is at most $|L_j||Y| \le (2/a)|X||Y|$.

		We attach $R$ to the larger one of the sets $X,Y$. For the following
		calculation we may assume that $|X| \ge |Y|$, in which case we consider the
		partition of $V(H)$ into sets $X \cup R, Y$. Since $|X \cup R| \ge |H| / 2$
		and $|R| \le (\zeta / 6) |H| \le (\zeta / 3) |X \cup R|$, the number of
		\crossedges{R}{Y} edges is at most $(\zeta / 3) |X \cup R||Y|$. Hence, the
		number of \crossedges{(X \cup R)}{Y} edges does not exceed $(2/a)|X||Y| +
		(\zeta/3)|X \cup R||Y| \le (2/a + \zeta/3)|X \cup R||Y|$. Therefore, we have
		$2/a + \zeta/3 \ge \zeta$, which implies that $a \le 3 / \zeta$, as desired.
	\end{proof}

\end{document}